\documentclass[11pt]{amsart}
\usepackage{geometry}                % See geometry.pdf to learn the layout options. There are lots.
\usepackage{graphicx}
\usepackage{amssymb}
\usepackage{epstopdf}
\newtheorem{theorem}{Theorem}[section]
\newtheorem{lemma}{Lemma}[section]
\newtheorem{corollary}{Corollary}[section]

\newtheorem{conjecture}{Conjecture}[section]

\numberwithin{equation}{section}

\def\Z{\mathbb Z}

\def\R{\mathbb R}

\def\d{\partial}
\def\a{\alpha}
\def\b{\beta}
\def\g{\gamma}
\def\s{\sigma}
\def\e{\epsilon}
\def\D{\Delta}

\title{Complexity of Shadows \& Traversing Flows in Terms of the Simplicial Volume}
\author{Gabriel Katz}

\address{5 Bridle Path Circle, Framingham, MA 01701, USA}
\email{gabkatz@gmail.com}
\begin{document}
\maketitle
%\section{}
%\subsection{}

\begin{abstract} 
We combine Gromov's amenable localization technique with the Poincar\'{e} duality to study the traversally generic vector flows on smooth compact manifolds $X$ with boundary.  Such flows generate well-understood stratifications of $X$ by the trajectories that are tangent to the boundary in a particular  canonical fashion. Specifically, we get  lower estimates of the numbers of connected components of these flow-generated strata of any given codimension. These universal  bounds are basically expressed in terms of the normed homology of the fundamental groups $\pi_1(X)$ and $\pi_1(DX)$, where $DX$ denotes the double of $X$. The norm here is the Gromov simplicial semi-norm  in homology. It turns out that some close relatives of the normed homology spaces $H_{k+1}(DX; \R)$, $H_{k}(X; \R)$ form  obstructions to the existence of $k$-convex traversally generic vector flows on $X$. 
\end{abstract} 

\section{Introduction}

This paper is a direct extension of \cite{AK}. As the latter article, it draws its inspiration from the paper of Gromov \cite{Gr1} where, among other things, the machinery of amenable localization has been developed.  Here we combine the amenable localization  with the Poincar\'{e} duality to study the \emph{traversally generic} vector flows (see \cite{K2} or Section 2 for the definition) on smooth compact manifolds $X$ with boundary.  

An example of a traversally generic field $v$  on a surface $X$ is shown in Fig. 1 (the field $v$ is vertical). The trajectory space $\mathcal T(v)$ of the $v$-flow is a graph whose verticies are trivalent or univalent. So the \emph{local} topology of $\mathcal T(v)$ is quite rigid; moreover, it is \emph{universal} for all  traversally generic fields on surfaces. The fibers of the obvious map $\Gamma: X \to \mathcal T(v)$ are closed segments or singletons, hence $\Gamma$ is a homotopy equivalence. Note that the trivalent verticies of $\mathcal T(v)$ correspond the the $v$-trajectories $\g$ that first pierce the boundary $\d X$ transversally, then tangentially touch it, then pierce it again traversally. We use the pattern $(121)$ to encode this behavior of $\g$. The majority of $\gamma$'s intersect $\d X$ at two points where $\g$ is transversal to $\d X$; we use the pattern $(11)$ to mark such $\gamma$'s. Finally, there are points-trajectories that correspond to the univalent verticies of $\mathcal T(v)$; they are marked with the combinatorial pattern $(2)$. The traversally generic fields on surfaces do not admit other combinatorial patterns of tangency (say, like $(31)$ or $(1221)$).

\begin{figure}[ht]
\centerline{\includegraphics[height=2.5in,width=3.3in]{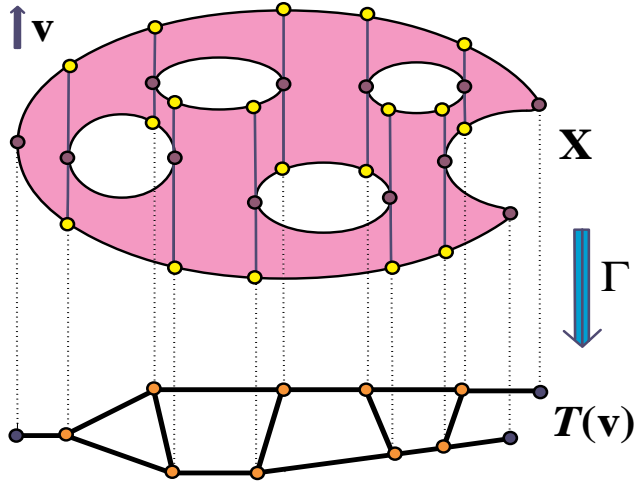}} 
\bigskip
\caption{Traversally generic ``vertical" field $v$ on the surface $X$, its stratified trajectory space $\mathcal T(v)$, and the obvious map $\Gamma: X \to \mathcal T(v)$}
\end{figure}

Similarly, for a smooth traversally generic vector field $v$ on a compact $(n+1)$-manifold $X$ with boundary, the trajectory space $\mathcal T(v)$  acquires a stratification $\{\mathcal T(v, \omega)\}_\omega$ by the combinatorial patterns of tangency $\omega$ that belong to a \emph{universal} poset $\Omega^\bullet_{'\langle n]}$  (see \cite{K2} and Section 2).  It depends only on $\dim(X)$. 

The more numerous the connected components of these stratifications are, the more \emph{complex} the $v$-flow is (thus the word ``complexity" in the title of the paper). So our goal here is to find  some \emph{lower} bounds of the numbers such components. \smallskip

The $\Omega^\bullet_{'\langle n]}$-stratification of  $\mathcal T(v)$ generates the stratification $\{X(v, \omega) =_{\mathsf{def}} \Gamma^{-1}(\mathcal T(v, \omega))\}_\omega$ of $X$ and the stratification $\{\d X(v, \omega) =_{\mathsf{def}} X(v, \omega) \cap \d X\}_\omega$ of $\d X$. The $X$-stratification can be refined by considering the connected components of the sets $\{\d X(v, \omega),\; X^\circ(v, \omega)\}_\omega\}_\omega,$ where $X^\circ(v, \omega) =_{\mathsf{def}} X(v, \omega) \setminus \d X(v, \omega)$. 
In Fig. 1, the $0$-dimensional strata of this stratification are the bold (light and dark) dots, the $1$-dimensional strata are the segments that connect the dots (some of them are portions of trajectories, some are arcs that belong to the boundary $\d X$), and  the $2$-dimensional strata are open cells in which the $1$-dimensional strata divide the surface $X$. \smallskip

We consider an auxiliary closed manifold, the double $DX =_{\mathsf{def}} X \cup_{\d X} X$ of $X$. The double comes equipped with an involution $\tau$ so that $(DX)^\tau = \d X$ and $DX/\{\tau\} = X$. 

We stratify $DX$ by the connected components of the sets (see Fig. 3): $$\{\d X(v, \omega),\; X^\circ(v, \omega),\; \tau(X^\circ(v, \omega)) \}_\omega$$

All these $v$-induced stratifications of $\mathcal T(v)$, $X$, and $DX$ are the foci of our investigation. \smallskip

Let $Z$ be a topological space. Recall that the  homology  $H_j(Z; \R)$ comes equipped with the Gromov \emph{simplicial semi-norm} $\| \sim \|_\D$ (see \cite{Gr} and Definition 3.3). We denote by $H_j^{\mathbf \D}(Z; \R)$ the quotient of $H_j(Z; \R)$ by the subspace of elements whose simplicial semi-norm is zero. Thus, $H_j^{\mathbf \D}(Z; \R)$ is a normed space with respect to the quotient semi-norm. \smallskip

For technical reasons, related to the application of the Gromov Localization Lemma \ref{lem11.1}, we substitute the relative homology $H_\ast(X, \d X)$ with the absolute homology $H_\ast(DX)$. \smallskip

By Theorem 4.5, $\dim\big(H_{k+1}^{\mathbf \D}(X; \R)\big) \leq C_{n-k}(v)$, where  $C_{n-k}(v)$ denotes the number of connected components of dimensions $n - k$ of the strata $\{X^\circ(v, \omega)\}_\omega$. 

Similarly, $\dim\big(H_{k+1}^{\mathbf \D}(DX; \R)\big) \leq \Sigma C_{n-k}(v)$, the number  of connected components of dimensions $n - k$ of the strata $\{\d X(v, \omega),\; X^\circ(v, \omega),\; \tau(X^\circ(v, \omega)) \}_\omega.$

Thus,  for \emph{any} traversally generic $v$-flow  on $X$, $\dim\big(H_{k+1}^{\mathbf \D}(X; \R)\big)$ and $\dim\big(H_{k+1}^{\mathbf \D}(DX; \R)\big)$ ---two homotopy-theoretical invariants of  $(X, \d X)$---deliver  lower estimates for $C_{n-k}(v)$ and $\Sigma C_{n-k}(v)$, the  $(n-k)$-dimensional \emph{traversing complexities} of $v$!  By their very nature, $C_{n-k}(X)$ and $\Sigma C_{n-k}(X)$---the \emph{minima} of  such  complexities, taken over all traversally generic vector fields $v$ on $X$,---are new invariants of the \emph{smooth} structure on $X$. So, the homotopy theory of the pair $(X, \d X)$ puts nontrivial restrictions on its  traversing complexity. %\smallskip

We stress that all our results are vacuous  when the fundamental groups $\pi_1(DX)$ or $\pi_1(X)$ are \emph{amenable} (see Definition 3.5). \smallskip %REF

Let us describe another geometrical manifestation of basically the same phenomenon (see Theorem 4.2 and Corollary 4.1). Let $f: M \to X$ be a map from a closed $(k +1)$-dimensional pseudo-manifold $M$ to $X$. It realizes the homology class $f_\ast[M] \in H_{k+1}(X)$. We may assume that $f$ is transversal to each pure $(n-k)$-dimensional  stratum from the stratification $\{X^\circ(v, \omega)\}_\omega$ of $\textup{int}(X)$. Then there exists an universal constant $c > 0$ (which depends only on $n$ and $k$) such that, for any $X$,  any traversally generic $v$ on it, and any $f$ as above, the intersection number of the cycle $f(M)$ with the union of $(n -k)$-dimensional strata from the stratification $\{X^\circ(v, \omega)\}_\omega$ is greater then or equal to $c \cdot  \|f_\ast[M]\|_\D$. %Here $j_\ast: H_k(X) \to H_k(X, \d X)$ denotes the natural homomorphism.
\smallskip

Here is a couple of \emph{examples} that show how these general conclusions apply to the traversally generic  flows on  smooth compact $4$-folds with boundary. When $\dim(X)= 4$, the only nontrivial lower bounds of traversing complexities can be provided by the groups $H^{\mathbf \D}_0(X), H^{\mathbf \D}_2(X), H^{\mathbf \D}_3(X)$, and $H^{\mathbf \D}_0(DX), H^{\mathbf \D}_2(DX), H^{\mathbf \D}_3(DX), H^{\mathbf \D}_4(DX)$.

The basic arguments that validate the following example are similar to the ones we use in Example 4.2.

Let $Z_i$ be a fibration over a closed oriented surface $M_i$ of a genus $g(M_i) \geq 2$ and with a  closed surface fiber $F_i$ ($i = 1, \dots, N$). Assume that  $Z_i \to M_i$ admits a section $s_i$. 
Form the connected sum $Z = Z_1\, \#\,  Z_2 \, \# \, \dots \, \# \,  Z_N$
and consider a smooth closed $4$-fold $W$ which is \emph{homotopy equivalent} to $Z$. 

By deleting the standard $4$-ball from $W$ we get the $4$-fold  $X = W \setminus D^4$ whose boundary is the sphere $S^3$. When at least one the fibers $\{F_i\}$ has the genus $g(F_i) \geq 2$, then $\|[DX]\|_\D > 0$ and $\dim\big(H_4^{\mathbf \D}(DX)\big) = 1$. Under these hypotheses, our results imply that  any transversally generic $v$ on $X$ must have at least one trajectory $\g$ from the following list: {\bf (a)} either $\g$ pierces he boundary sphere transversally, then simply touches it 3 times, then transversally pierces it again (the  combinatorial pattern of such $\g$ is $(12221)$), or {\bf (b)}  $\g$ is cubically tangent to the boundary, then simply tangent, then pierces $S^3$ transversally (the  combinatorial pattern of such $\g$ is either $(321)$ or $(123)$), or {\bf (c)} $\g$ is transversal to the boundary sphere, then quartically tangent to it, then meets it transversally again (the  combinatorial pattern of such $\g$ is $(141)$). Moreover, we prove that there exists a universal constant $\theta > 0$, such that the number of such trajectories grows as $\theta \cdot \|[DX]\|_\D$ at least.\smallskip

The images of the classes $\{[s_i(M_i)]\}_i$ are independent in $H^{\mathbf \D}_2(X)$. Thus, $\dim\big(H_2^{\mathbf \D}(X)\big) \geq N$. Applying  Theorem \ref{th11.10}  to the traversally generic $v$ on $X$, we get 
\[
3 \cdot\#\pi_0(\mathcal T(v, 1221)) + \#\pi_0(\mathcal T(v, 13)) + \#\pi_0(\mathcal T(v, 31)) \geq N,  
\]
where $\#\pi_0(\mathcal T(v, \omega))$ denotes the number of connected $1$-dimensional components of the combinatorial type $\omega$ in $\mathcal T(v)$.
%%%%

The Morse theory helps to exclude a priory the tangencies of local multiplicities $3$ and higher. Let  $W$ be as before and let $f: W \to \R$ be a Morse function whose gradient field $v$ satisfies the Morse-Smale transversality property. We form a compact smooth $4$-fold $X$ by deleting from $W$ sufficiently small standard $4$-balls, centered on the $f$-critical points. For such a choice of $X$ and $v$, the combinatorial tangency patterns of the $v$-trajectories in $X$ belong to the list: $(11), (121), (1221), (12221)$. Then by similar arguments, $3 \cdot\#\pi_0(\mathcal T(v, 1221)) \geq N$.  \hfill $\diamondsuit$
 \smallskip

Now let us describe the structure of the paper.  In Section 2, for the reader's convenience, we reintroduce the notion of a traversally generic vector field and describe its basic properties, needed for Section 4. (they have been studied in a series of papers \cite{K}, \cite{K1} -\cite{K4}, and \cite{AK}). 
\smallskip

In Section 3, we study maps from a given compact $\mathsf{PL}$-manifold $X$ with boundary onto special compact $CW$-complexes $K$, $\dim(K) = \dim(X) - 1$. The local topology (the types of singularities) of $K$ is prescribed a priori; it is $X$-independent. We require the fibers of $F: X \to K$ to be $\mathsf{PL}$-homeomorphic to closed segments  or to singletons. We call such maps $F$ the \emph{shadows} of $X$. This setting is mimicking the maps $\Gamma: X \to \mathcal T(v)$, generated by traversally generic fields $v$ on smooth manifolds $X$ with boundary. %; here $\Gamma$ takes each point $x \in X$ to the $v$-trajectory $\g_x$ through $x$. 

The target space $K$ of a shadow $F$ comes equipped with a natural stratification, defined by the local topology of the singular loci in $K$. With the help of $F$, that stratification induces stratifications in $X$ and in its double $DX$. 

We introduce the $j$-th \emph{complexities} of a shadow $F$ as the number of connected components of the strata of the fixed dimension $j$ in $F(X)$, $X$ or in $DX$.

The main results of Section 3 are Theorem 3.1-3.3.  The first two deal with ``the amenable localization of the Poincar\'{e} duality operators", in particular, with estimations of their norms in terms of the normed ``homology" groups $H^{\mathbf\D}_\ast(DX; \R)$ or $H^{\mathbf\D}_\ast(X; \R)$. %They link the singularities of shadows $F: X \to K$ with the  ``homology" $H^{\mathbf\D}_\ast(DX; \R)$ or $H^{\mathbf\D}_\ast(X; \R)$. %, and the spaces $H^{\mathbf\D}_\ast(DX; \R)$ are close relatives of the normed homology spaces $H_\ast(DX; \R)$. 
%The norm here is the Gromov simplicial semi-norm  in homology. 
In Theorem 3.3,  the $j$-th complexities of any shadow $F$ are estimated from below by the ranks of the groups $H^{\mathbf\D}_{n+1-j}(DX; \R)$ or $H^{\mathbf\D}_{n+1-j}(X; \R)$, where $n+1 = \dim(X)$. 
\smallskip

In Section 4, we apply the results from Section 3 to the special shadows, produced by the traversally generic vector fields. The applications deliver two main results, Theorem 4.2 and Theorem 4.5. Recall that the $v$-flow canonically generates some well-understood stratifications of the spaces $\mathcal T(v)$, $X$, and $DX$ (see \cite{K2}). As in the category of shadows,  these stratifications lead to few competing notions of \emph{complexity} for traversally generic flows. In Theorem 4.5, we get  lower estimates of the numbers of connected components of these flow-generated strata of any given dimension. The estimates are \emph{universal} for the given homotopy type of the pair $(X, \d X)$ and any traversing field on $X$. Again, these universal  bounds are expressed in terms of the \emph{normed} homology of $DX$ or of $X$. Moreover, we prove that  that  the normed spaces $H^{\mathbf\D}_{n+1 -j}(DX; \R)$ and  $H^{\mathbf\D}_{n -j}(X; \R)$ form \emph{obstructions} to the existence of the \emph{globally $j$-convex} (see Definition 2.2) traversally generic vector flows on a given $X$.

In the process of studying various  complexities  of traversally generic flows, we introduce few graded differential complexes $\mathbf C_\ast^\mho(\mathcal T(v))$, $\mathbf C_\ast^\mho(X, v)$, $\mathbf C_\ast^\mho(DX, v)$ (see formulae (\ref{eq11.30}), (\ref{eq11.34})). They are naturally produced by the filtrations of the spaces $\mathcal T(v)$, $X$,  and $DX$ by the $v$-induced strata of a fixed codimension. The differential complexes $\mathbf C_\ast^\mho(\sim, v)$ are more refined invariants of the $v$-flow than the flow complexities: the complexities $\{tc_j(\sim, v)\}_j$ are just the ranks of the corresponding $j$-graded terms of these complexes. Although $\mathbf C_\ast^\mho(\sim, v)$ seem to be the right instruments for studying traversing vector fields on $X$ and ultimately $X$ itself, their homological investigation belongs to a different paper. 

%%%%%
%%%%
\section{Basics of Traversally Generic Vector Fields}

We start with presenting few basic definitions and facts related to the traversally generic vector fields.

Let $X$ be a compact connected smooth $(n+1)$-dimensional manifold with boundary. A vector field $v$ is called \emph{traversing} if each $v$-trajectory is ether a closed interval with both ends residing in $\d X$, or a singleton also residing in $\d X$ (see \cite{K1} for the details). In particular, a traversing field does not vanish in $X$. In fact, $v$ is traversing if and only if $v \neq 0$ and $v$ is of the gradient type (see \cite{K1}). 

For traversing fields $v$, the trajectory space $\mathcal T(v)$ is homology equivalent  to $X$ (Theorem 5.1, \cite{K3}).   \smallskip

We denote by $\mathcal V_{\mathsf{trav}}(X)$ the space of traversing fields on $X$. \smallskip

In this paper, we consider an important subclass of traversing fields which we call \emph{traversally generic} (see formula (2.4) and Definition 3.2 from \cite{K2}).

For a traversally  generic field $v$, the trajectory space $\mathcal T(v)$ is stratified by closed subspaces, labeled by the elements $\omega$ of an \emph{universal} poset $\Omega^\bullet_{'\langle n]}$  which depends only on $\dim(X) = n+1$ (see \cite{K3}, Section 2, for the definition and properties of $\Omega^\bullet_{'\langle n]}$). The elements $\omega \in \Omega^\bullet_{'\langle n]}$ correspond to combinatorial patterns that describe the way in which $v$-trajectories $\g \subset X$ intersect the boundary $\d_1 X =_{\mathsf{def}} \d X$. Each intersection point $a \in \g \cap \d_1X$ acquires a well-defined \emph{multiplicity} $m(a)$, a natural number that reflects the order of tangency of $\g$ to $\d_1X$ at $a$ (see \cite{K1} and Definition 2.1 for the expanded definition of $m(a)$). So $\g \cap \d_1X$ can be viewed as a \emph{divisor} $D_\g$ on $\g$, an ordered set of points in $\g$ with their multiplicities. Then $\omega$ is just the ordered sequence of multiplicities $\{m(a)\}_{a \in \g \cap \d_1X }$, the order being prescribed by $v$. 

The support of the divisor $D_\g$ is either a singleton $a$, in which case $m(a) \equiv 0 \; \mod \, 2$, or the minimum and maximum points of $\sup D_\g$ have \emph{odd} multiplicities, and the rest of the points have \emph{even} multiplicities. \smallskip

Let 
\begin{eqnarray}\label{eq2.1}
m(\g) =_{\mathsf{def}} \sum_{a \in \g\, \cap \, \d_1X }\; m(a)\quad \text{and} \quad m'(\g) =_{\mathsf{def}} \sum_{a \in \g \, \cap \, \d_1X }\; (m(a) -1).
\end{eqnarray} 
Similarly, for $\omega =_{\mathsf{def}} (\omega_1, \omega_2, \dots , \omega_i,  \dots )$ we introduce the \emph{norm} and the \emph{reduced norm} of $\omega$ by the formulas: 

\begin{eqnarray}\label{eq2.2}
|\omega| =_{\mathsf{def}} \sum_i\; \omega_i \quad \text{and} \quad |\omega|'  =_{\mathsf{def}} \sum_i\; (\omega_i -1).
\end{eqnarray}
\smallskip

Let $\d_jX =_{\mathsf{def}} \d_jX(v)$ denote the locus of points $a \in \d_1X$ such that the multiplicity of the $v$-trajectory $\g_a$ through $a$ at $a$ is greater than or equal to $j$. This locus has a description in terms of  an auxiliary function $z: \hat X \to \R$ which satisfies the following three properties:
\begin{eqnarray}\label{eq2.3}
\end{eqnarray}

\begin{itemize}
\item $0$ is a regular value of $z$,   
\item $z^{-1}(0) = \d_1X$, and 
\item $z^{-1}((-\infty, 0]) = X$. 
\end{itemize}

In terms of $z$, the locus $\d_jX =_{\mathsf{def}} \d_jX(v)$ is defined by the equations: 
$$\{z =0,\; \mathcal L_vz = 0,\; \dots, \;  \mathcal L_v^{(j-1)}z = 0\},$$
where $\mathcal L_v^{(k)}$ stands for the $k$-th iteration of the Lie derivative operator $\mathcal L_v$ in the direction of $v$ (see \cite{K2}). 

The pure stratum $\d_jX^\circ \subset \d_jX$ is defined by the additional constraint  $\mathcal L_v^{(j)}z \neq 0$. %The locus $\d_jX$ is the union of two loci: $(1)$ $\d_j^+X$, defined by the constraint  $\mathcal L_v^{(j)}z \geq  0$, and $(2)$ $\d_j^-X$, defined by the constraint  $\mathcal L_v^{(j)}z \leq  0$. The two loci, $\d_j^+X$ and $\d_j^-X$, share a common boundary $\d_{j+1}X$.
\smallskip

\noindent{\bf Definition 2.1} The multiplicity $m(a)$, where $a \in \d X$, is the index $j$ such that $a \in \d_jX^\circ$. 

\hfill $\diamondsuit$
\smallskip

The characteristic property of \emph{traversally generic} fields is that they admit special flow-adjusted coordinate systems, in which the boundary is given by  quite special polynomial equations (see formula (\ref{eq2.4})) and the trajectories are parallel to one of the preferred coordinate axis (see  \cite{K2}, Lemma 3.4). For a traversally generic $v$ on a $(n+1)$-dimensional $X$, the vicinity $U \subset \hat X$ of each $v$-trajectory $\g$ of the combinatorial type $\omega$ has a special coordinate system $$(u, \vec x, \vec y): U \to \R\times \R^{|\omega|'} \times \R^{n-|\omega|'}.$$ By Lemma 3.4  and formula $(3.17)$ from \cite{K2}, in these coordinates, the boundary $\d_1X$ is given  by the polynomial equation: 
\begin{eqnarray}\label{eq2.4}
 \wp (u, \vec x) =_{\mathsf{def}} \prod_i \big[(u-i)^{\omega_i} + \sum_{l = 0}^{\omega_i-2} x_{i, l}(u -i)^l \big] = 0
 \end{eqnarray}
of an even degree $|\omega|$ in $u$. Here  $i \in \Z$ runs over the distinct roots of  $\wp (u, \vec 0)$ and  $\vec x =_{\mathsf{def}} \{ x_{i, l}\}_{i,l}$.   
At the same time, $X$ is given by the polynomial inequality $\{\wp(u, \vec x) \leq 0\}$.  Each $v$-trajectory in $U$ is produced by freezing all the coordinates $\vec x, \vec y$, while letting $u$ to be free.
\smallskip

We denote by $X(v, \omega)$ the union of $v$-trajectories whose divisors are of a given combinatorial type $\omega \in \Omega^\bullet_{'\langle n]}$. Its closure $\cup_{\omega' \preceq_\bullet \omega}\; X(v, \omega')$ is denoted by $X(v, \omega_{\succeq_\bullet})$.
\smallskip

Each pure stratum $\mathcal T(v, \omega) \subset \mathcal T(v)$ is an open smooth manifold and, as such, has a ``conventional" tangent bundle. \smallskip

We denote by $\mathcal V^\ddagger(X)$ the space of traversally  generic fields on $X$.  It turns out that $\mathcal V^\ddagger(X)$ is an \emph{open} and \emph{dense} (in the $C^\infty$-topology) subspace of $\mathcal V_{\mathsf{trav}}(X)$ (see \cite{K2}, Theorem 3.5).
\smallskip

\noindent{\bf Definition 2.2.} We say that a traversing field $v$ on $X$ is \emph{globally} $k$-\emph{convex} if $m'(\g) < k$ for any $v$-trajectory $\g$. \hfill $\diamondsuit$
\smallskip

%%%%%% 
\section{Shadows of Manifolds with Boundary and their Complexity}
%%%%%%

The notion and properties of \emph{shadows} (see Definition 3.1), the main subject of this section, are inspired by the maps $\Gamma: X \to \mathcal T(v)$, where the field $v$ is traversally generic. \smallskip

We are going to pick a fixed and carefully chosen class of compact $n$-dimensional $CW$-complexes $K$ whose local topological structure is prescribed. 

Let $X$ be a compact connected $\mathsf{PL}$-manifold of dimension $n+1$ with boundary. 

We will consider a variety of surjective maps $\{F: X \to K\}$ with the $F$-fibers being a particular type of \emph{contractible} $1$-dimensional complexes (in this paper, segments and singletons). We think of such $CW$-complexes $K = F(X)$ as ``shadows" of the given manifold $X$. We consider singularities in $K$ of particular types $\{K(\omega)\}_\omega$ and intend to count the cardinalities $\{\#\pi_0(K(\omega))\}_\omega$.  We  view this count of connected components of the strata $K(\omega)$ as measuring the \emph{complexity} of the surjection $F$. Then we minimize these numbers over all $F$'s to get various notions of complexity for the given manifold $X$. \smallskip

Let us start with a quite general setting. Let $\mathcal S$ be a poset equipped with two maps:  a map $\mu: \mathcal S \to \Z_+$ and an order-preserving map $\mu': \mathcal S \to \Z_+$. By definition, for each $\omega \in \mathcal S$, $\mu'(\omega) < \mu(\omega)$. For each $n \in \Z_+$, we assume that the poset $\mathcal S_n =_{\mathsf{def}} (\mu')^{-1}([0, n])$ is finite. 

With each element $\omega \in \mathcal S_n$ we associate a model compact $CW$-complex $\mathsf T_\omega$ of dimension $\mu'(\omega)$ and a model compact $\mathsf{PL}$-manifold $\mathsf E_\omega$ of dimension $\mu'(\omega) +1$. They are linked by a $\mathsf{PL}$-map $p_\omega: \mathsf E_\omega \to \mathsf T_\omega$ whose fibers are \emph{closed intervals} or \emph{singletons}.  In what follows, we will list additional properties of the two collections, $\mathsf E =_{\mathsf{def}} \{\mathsf E_\omega\}_{\omega \in \mathcal S}$ and  $\mathsf T =_{\mathsf{def}} \{\mathsf T_\omega\}_{\omega \in \mathcal S}$ (exhibiting topologically distinct $\mathsf T_\omega$'s). We will do it in a recursive fashion.
\smallskip

Consider a set $\mathsf {Cosp}(\mathsf T, n)$\footnote{``$\mathsf {Cosp}$" is an abbreviation of ``cospine".} of $n$-dimensional compact $CW$-complexes $K$ such that each point $y \in K$ has a neighborhood which is $\mathsf{PL}$-homeomorphic to the product $\mathsf T_\omega \times D^{n- \mu'(\omega)}$ for some $\omega \in \mathcal S$, where $\mu'(\omega) \leq n$, and  $D^{n- \mu'(\omega)}$ denotes the standard ball. 

We require that each model space $\mathsf T_\omega$\footnote{``normal" to the $\omega$-labeled stratum $K(\omega)$ in $K$} topologically will be a \emph{cone} over a space $\mathsf S_\omega$ that belongs to the set $\mathsf {Cosp}(\mathsf T, n-1)$. 

By definition, $\mathsf {Cosp}(\mathsf T, 1)$ consists of finite graphs whose verticies are of valencies $1$ and $3$ only.

We denote by $K(\omega)$ the set of points in $K$ whose neighborhoods are modeled after  the space $\mathsf T_\omega \times D^{n- \mu(\omega)}$. It follows that each $K(\omega) \subset K$ is a locally closed  $\mathsf{PL}$-manifold.  

Let us also consider a filtration 
$K = K_{-0} \supset K_{-1} \supset \dots \supset K_{-n}$
of $K$ by the closed subcomplexes
\begin{eqnarray}\label{eq11.2}
K_{-j} =_{\mathsf{def}} \bigcup_{\{\omega \in \mathcal S_n |\; \mu'(\omega) \geq j\}} \, K(\omega)
\end{eqnarray}
of dimensions $n-j$.  Note that $K_{-j} = \emptyset$ implies $K_{-(j+1)} = \emptyset$.
\smallskip

%\begin{definition}
{\bf Definition 3.1.}
Let $X$ be a compact connected $\mathsf{PL}$-manifold of dimension $n+1$.  We assume that $\d X \neq \emptyset$. Consider the set  $\mathsf {Shad}(X, \mathsf E \Rightarrow \mathsf T)$ of surjective $\mathsf{PL}$-maps $F: X \to  K$ such that:
\begin{itemize}
\item $K$ is a compact $CW$-complex of the type from $\mathsf {Cosp}(\mathsf T, n)$,
\item each fiber of $F$ is $\mathsf{PL}$-homeomorphic to a closed  interval $I$ or to a singleton $pt$, where $\partial I$ and $pt$ reside in $\partial X$.
\item $F^\partial =_{\mathsf{def}} F|: \partial X \to K$ is a surjective map with finite fibers,
\item for each $\omega \in \mathcal S$, the map  $$F|: F^{-1}(K(\omega)) \to K(\omega)$$ is a trivial fibration with an orientable manifold base, and the  map $$F^\partial |: (F^\partial)^{-1}(K(\omega)) \to K(\omega)$$ is a trivial covering with the fiber of cardinality $\mu(\omega) -\mu'(\omega)$,
\item each point $y \in K(\omega)$ has a regular neighborhood $V_y \subset K$ such that the map $$F: (F^{-1}(V_y), F^{-1}(\d V_y)) \to (V_y, \d V_y)$$ is conjugate to the model map $$p_\omega \times id : (\mathsf E_\omega, \delta \mathsf E_\omega) \times D^{n- \mu(\omega)} \to  (\mathsf T_\omega, \mathsf S_\omega) \times D^{n- \mu(\omega)}$$ via a $\mathsf{PL}$- homeomorphism which preserves the $\mathcal S_n$-stratifications in both spaces. Here $\delta \mathsf E_\omega$ denotes the portion of the boundary $\d \mathsf E_\omega$ that is mapped to $\mathsf S_\omega$.
\end{itemize}

For $F \in \mathsf {Shad}(X,  \mathsf E \Rightarrow \mathsf T)$, we call the $CW$-complex  $F(X)$ a $\mathcal S$-\emph{shadow} of $X$. 
\hfill $\diamondsuit$
%\end{definition}
\smallskip

{\bf Definition 3.2.}
%\begin{definition}\label{def11.7} 
We say that  a compact connected $\mathsf{PL}$-manifold $X$ with boundary is \emph{globally $k$-convex} if it has a shadow $F \in  \mathsf {Shad}(X,  \mathsf E \Rightarrow \mathsf T)$ with the property $F(X)_{-k} = \emptyset$.

Note that the global $k$-convexity implies global $(k+1)$-convexity. \hfill $\diamondsuit$
%\end{definition}
\smallskip

Part of this section is devoted to finding obstructions to the global $k$-convexity, selected from the tool set of algebraic topology. \smallskip

Next, we will employ \emph{Gromov's simplicial semi-norms} \cite{Gr} to give lower estimates of  the complexities of  shadows and of traversing flows on a given $(n+1)$-manifold $X$. The number of $v$-trajectories of the (maximal) reduced multiplicity $\dim(X) - 1$ can serve as an example of such complexity. 
First, let us recall a few relevant definitions (see \cite{Gr}, \cite{Gr1}). \smallskip

{\bf Definition 3.3.}
%\begin{definition}\label{def11.9} 
Let $X \supset Y$ be topological spaces. Given a relative homology class $h \in H_k(X, Y; \R)$, consider all relative singular cycles $c = \sum_i r_i \s_i$ that represent $h$. Here $r_i \in \R$ and $\s_i: \D^k \to X$ are singular simplicies;  each $(k-1)$-simplex from the algebraic boundary $\d c$ being mapped to $Y$. We assume that, for any compact $K \subset X$, only finitely many images $\{\s_i(\D)\}$ intersect $K$. 
Put $$\|c\|_{l_1} =_{\mathsf{def}} \sum_i |r_i|.$$ We define the \emph{simplicial semi-norm} of a homology class $h$ by the formula:
$$\| h\|_\D =_{\mathsf{def}} \inf_c \big\{\|c\|_{l_1}\big\}.$$
%Similar definition of $\| h\|_\D$ is available for relative homology classes $h \in H_k(X, Y; \R)$, where $Y \subset X$ is a subspace. 
\hfill $\diamondsuit$
%\end{definition}
\smallskip

This construction defines a semi-norm $\|\sim \|_\D$ on the vector space $H_k(X, Y;\, \R)$. The semi-norm is monotone decreasing under continuous maps of pairs of spaces: $$\|h\|_\D \geq  \|f_\ast(h)\|_\D$$ for any $h \in H_k(X_1, Y_1\; \Z)$ and a continuous map $f: (X_1, Y_1) \to (X_2, Y_2)$.
Moreover, if $f: X_1 \to X_2$ is a continuous map such that $f_\ast: \pi_1(X_1) \to \pi_1(X_2)$ is an \emph{isomorphism} of the fundamental groups, then $\| f_\ast(h) \|_\D = \| h \|_\D$ for any $h \in H_k(X_1; \R)$ \cite{Gr}.
\smallskip

If $M$ is any closed, oriented \emph{hyperbolic} manifold, then 
\begin{eqnarray}\label{eq11.3}
\textup{Vol}(M) = c(n) \cdot \| [M] \|_\D,
\end{eqnarray}
where $[M]$ denotes the fundamental class of $M$ and $c(n)$ is an universal positive constant (this is the Proportionality Theorem, page 11 of [Gr]). For this reason, the simplicial norm of the fundamental class $[X]$ is often called the \emph{simplicial volume}.\smallskip

Gromov's   Localization Lemma 3.1 below relies on the notion of the \emph{stratified simplicial semi-norm},  available for stratified spaces $X$ and pairs $X \supset Y$ of stratified spaces. \smallskip

We consider stratified spaces such that if a stratum $S$ intersects the closure $\overline{S'}$ of another stratum $S'$, then $S \subseteq \overline{S'}$.  In this case we write $S \preceq S'$.  If neither $S \preceq S'$ nor $S' \preceq S$, then we say the two strata are \emph{incomparable}. \smallskip

Recall that the \emph{corank} of a stratum $Y_\omega$ in a $\mathcal S$-stratification of a given space $Y$ is the maximal length $k$ of a filtration $Y_\omega \subset \bar Y_{\omega_1} \subset \dots \subset \bar Y_{\omega_k}$ by the distinct strata whose closures contain $Y_\omega$. \smallskip 

\begin{figure}[ht]
\centerline{\includegraphics[height=3in,width=4in]{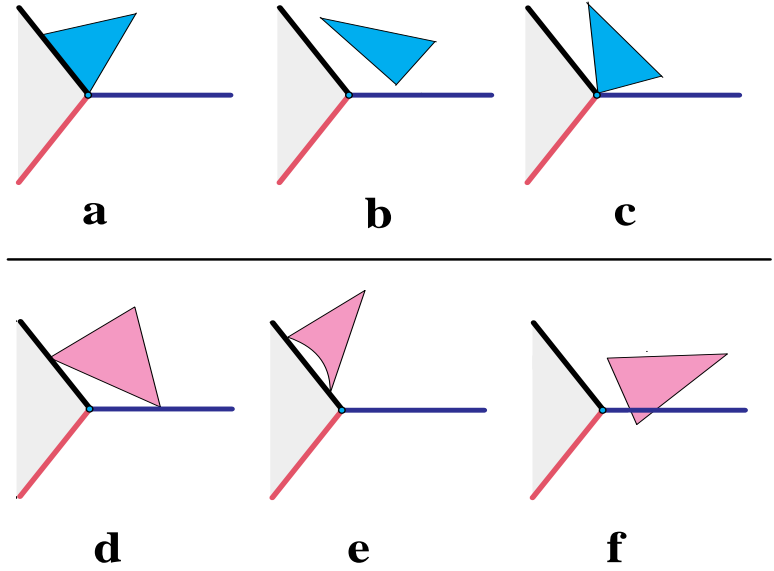}} 
\bigskip
\caption{Examples of a singular $2$-simplex in relation to a stratification of the plane by a single stratum of codimension $2$, three strata of codimension $1$,  and two strata of codimension $0$. Diagrams (a), (b), (c) are consistent with the four bullet list from Definition 3.4; diagram (d) violates the second bullet, diagram (e) violates the third bullet, diagram (f) violates the first bullet.}
\end{figure}

{\bf Definition 3.4.} 
%\begin{definition}\label{def11.10} 
Let $X$ be a $\mathcal S$-stratified topological space. Given a real homology class $h \in H_k(X; \R)$, consider all singular cycles $c = \sum_i r_i \s_i$ that represent $h$, where $r_i \in \R$ and $\s_i: \D^k \to X$ are singular simplicies that are consistent with the stratification $\mathcal S$, in the following sense\footnote{Gromov gives two conditions: \emph{ord}(er) and \emph{int}(ernality) (\cite{Gr1}, p.~27).  We use these two conditions plus two more!}:

\begin{itemize}
\item   We require that for each simplex $\s_i$ of $c$, the image of the interior of each face (of any dimension) must be contained in one stratum.  We call this the \emph{cellular} condition.
\item The (\emph{ord}) condition states that the image of each simplex of $c$ must be contained in a totally ordered chain of strata; that is, the simplex does not intersect any two incomparable strata.  %{\bf ????}
\item The (\emph{int}) condition states that for each simplex of $c$, if the boundary of a face (of any dimension) maps into a stratum $S$, then the whole face maps into $S$.  
\item For technical reasons (involving the Amenable Reduction Lemma in \cite{AK}), we require that if two vertices of a simplex  $\s_i$ map to the same point $v \in X$, then the edge between them must be constant at $v$.  We call this the \emph{loop} condition.
\end{itemize}
We define the $\mathcal S$-\emph{stratified} simplicial semi-norm of a homology class $h$ by the formula:
$$\| h\|_\D^{\mathcal S} =_{\mathsf{def}} \inf_c \big\{\|c\|_{l_1}\big\},$$
where $c$ runs over all the cycles $c = \sum_i r_i \s_i$ that represent $h$, subject to the four properties above.

Similar definition of $\| h\|_\D^{\mathcal S}$ is available for relative homology classes $h \in H_k(X, Y; \R)$, where the inclusion $Y \subset X$ is a $\mathcal S$-stratified map. \hfill $\diamondsuit$
%\end{definition}
\smallskip

{\bf Definition 3.5.} 
%\begin{definition}\label{def11.11} 
A discrete group $G$ is called \emph{amenable} if for every finite subset $S \subset G$ and every $\e > 0$, there exists a finite non-empty set $A \subset G$
such that the proportion of cardinalities $$\frac{|((g \cdot A) \cup A) \setminus ((g \cdot A) \cap A)|}{|A|} < \e$$ for all $g \in S$. \hfill $\diamondsuit$
%\end{definition}
\smallskip

Finally, we are in position to state the pivotal Gromov's  Localization Lemma from \cite{Gr1}, page 772. Its proof there is a bit rough; a detailed proof can be found in \cite{AK}.

\begin{lemma}[{\bf Gromov's  Localization Lemma}]\label{lem11.1}
Let $X$ be a closed $(n+1)$-manifold with stratification $\mathcal{S}$ consisting of finitely many connected locally closed submanifolds. Pick  an integer $j$ from the interval $[0, n+1]$.  Let $Z$ be a space with the contractible universal cover\footnote{that is, a $K(\pi, 1)$-space}, and let $\alpha : X \rightarrow Z$ be a continuous map such that the $\alpha$-image of the fundamental group of each stratum of codimension less than $j$ is an amenable subgroup of $\pi_1(Z)$.  

Let $X_{-j} \subseteq X$ denote the union of strata with codimension at least $j$, and let $U$ be a neighborhood of $X_{-j}$ in $X$.  Then the $\alpha$-image of every $j$-dimensional homology class $h \in H_j(X)$ satisfies the upper bound
\[\| \alpha_*(h) \|_{\Delta} \leq \| h_U\|_{\Delta}^\mathcal{S},\]
where $h_U \in H_j(U, \d U)$ denotes the restriction of $h$ to $U$, obtained via the composite homomorphism 
$$H_j(X) \rightarrow H_j(X, X \setminus U) \rightarrow H_j(U, \d U),$$
where the last map is the excision isomorphism. \hfill $\diamondsuit$
\end{lemma}

Let $X$ be a compact oriented manifold with boundary. For technical reasons related to the application of  Lemma \ref{lem11.1}, many arguments to follow will deal with the \emph{double} $DX = X\cup_{\d X} X$ of $X$ instead of the pair $(X, \d X)$. The orientation on $X$ extends to an orientation on its double. Then there is an orientation-reversing involution $\tau: DX \to DX$  whose orbit space is $X$ and whose fixed point set $DX^\tau = \d X$.

Note that any absolute homology class $\tilde h \in H_\ast(DX)$, via the restriction to $X$,  gives rise to a relative class $h \in H_\ast(X, \d X)$. Conversely, every relative class $h$, with the help of the involution $\tau$ gives rise to an absolute class $\tilde h  \in H_\ast(DX)$ whose restriction to $X \subset DX$ produces $h$. Therefore,  $H_\ast(X, \d X)$ can be viewed as a direct summand of $H_\ast(DX)$. %In what follows, we prefer to deal with $h$, rather with its close relative $\tilde h$ since the application of the Localization Lemma \ref{lem11.1} is more straightforward and the statements of theorems and corollaries  become more economic. 

However, the relation between the simplicial semi-norms of  $\tilde h  \in H_\ast(DX)$ and of $h \in H_\ast(X, \d X)$ is not so transparent. For example, if $X$ is a cylindrical surface and  $h \in H_1(X, \d X)$ is a generator, then $\|h\|_\D = 1$, while $\|\tilde h\|_\D = 0$ since homologically the longitude $\tilde h$ of the torus $DX$  can be represented by the singular rational cycle  $\frac{1}{N}\cdot \{f: [0, 1] \to DX\}$, where $f$ wraps the segment $[0, 1]$ $N$ times around the longitude; so  $\|\tilde h\|_\D \leq \frac{1}{N}$ for all $N$. 

Note that $X/\d X$ is a torus with one of its meridians being collapsed to a point. The simplicial semi-norm of the longitude in $X/\d X$ is zero.  So, although $H_1(X, \d X)$ is canonically isomorphic to the reduced homology $H_1(X/\d X, \d X/\d X)$, the isomorphism is not an isometry in the two simplicial semi-norms!

At the same time, similar considerations imply that, for the cylinder $X$, the simplicial semi-norms on both $H_2(X, \d X)$ and $H_2(DX)$ are trivial.
\smallskip

\begin{lemma}\label{lem11.2} Let $X$ be a connected compact and oriented $(n+1)$-manifold with boundary. Let $\b: DX \to X/\d X$ be the degree 1 map which collapses the second copy $\tau(X) \subset DX$ of $X$ to the point $\star = \d X/\d X \in X/\d X$, and let $\a: (X, \d X) \to (X/\d X, \star)$ be the obvious map of pairs. 
Then, for any $h \in H_j(X, \d X)$,  we get
$$2\|h\|_\D \geq \|h - \tau_\ast(h)||_\D \geq \|\a_\ast(h)\|_\D,$$
where the middle term is the simplicial semi-norm of the class $h - \tau_\ast(h) \in H_j(DX)$.

In particular, for the fundamental class  $h = [X,  \d X] \in H_{n+1}(X, \d X; \R)$, the class \hfill\break $h - \tau_\ast(h) = [DX]$ is fundamental; so we get $$2\| [X,  \d X] \|_\D \geq \| [DX] \|_\D \geq \|[X/\d X] \|_\D.$$
 \end{lemma}
 
\begin{proof} If $c = \sum_i r_i \s_i$ is a relative cycle with the $l_1$-norm $\|c\|$ being $\epsilon$-close to $\| h \|_\D$, then the chain $\sum_i r_i \s_i - \sum_i r_i \tau(\s_i)$ is an absolute cycle in $DX$ whose $l_1$-norm is $2 \| c\|$ at most. Thus $2\| h \|_\D \geq \| h - \tau_\ast(h)\|_\D$. 
 
On the other hand, consider the degree $1$ map $\b: DX \to X/\d X$. By \cite{Gr}, the simplicial volume does not increase under continuous maps. Therefore, observing that $\a_\ast(h) = \b_\ast(h - \tau_\ast(h))$, we get $\|  h - \tau_\ast(h)\|_\D \geq \| \a_\ast(h)] \|_\D.$
  \end{proof}
%\smallskip

Let $\mathcal S$ be a stratification of $X$ by strata $S$ such that the sets $\{S \cap \d X\}_{S \in \mathcal S}$ form a stratification $\mathcal S^\d$ of the boundary $\d X$. Then $\mathcal S$ gives rise to a stratification $\mathcal{DS}$ of the double $DX$ so that any stratum $S \in \mathcal{DS}$ either belongs to $\d X$ (in which case $S \in \mathcal S^\d$) or to $DX \setminus \d X$. Conversely, any $\tau$-invariant stratification $\mathcal{DS}$ of the double induces a stratification $\mathcal S$ of $X$. 
\smallskip

Next, we are going to associate few useful differential chain complexes and their homology groups with shadows of manifolds $X$ with boundary.  Eventually they will become useful instruments in our investigations of various notions of complexity of traversing flows.\smallskip

Let $F \in \mathsf{Shad}(X,  \mathsf E \Rightarrow \mathsf T)$ be a shadow (see Definition 3.1) of a compact orientable $\mathsf{PL}$-manifold $X$ of dimension $n+1$. Put $K =_{\mathsf{def}} F(X)$. We consider the finite $\mathcal S_n$-stratification  $\{K(\omega)\}_{\omega \in \mathcal S_n}$ of $K$ by the model (normal to the strata) spaces $\{\mathsf T_\omega\}$ of dimensions $\{\mu'(\omega)\}$.  We denote by $X(F, \omega)$ the $F$-preimage of the pure stratum $K(\omega)$. These sets form a stratification $\mathcal S_F(X)$ of $X$. 

The $CW$-complex $K \in \mathsf{Cosp}(\mathsf T, n)$ comes equipped with a filtration $$K =_{\mathsf{def}} K_{-0} \supset K_{-1} \supset \dots \supset K_{-n}$$ which has been introduced in (2.1) %(\ref{eq11.2}) 
and employs the more refined $\mathcal S_n$-stratification. 

Let $\mathcal A$ be an abelian coefficient system on $K$. As a default, $\mathcal A = \R$, the trivial coefficient system with the real numbers for a stalk. 

For each $j \in [0, n]$, consider the relative homology groups\footnote{in our notations, we suppress the dependance of these homology groups on the coefficients $\mathcal A$. }
\begin{eqnarray}\label{eq11.4} 
\{\mathsf C^\mho_j(K) : = H_j(K_{-n+j}, K_{-n+j-1}; \, \mathcal A)\}_j  
\end{eqnarray}
associated with the filtration. Note that $\dim(K_{-n+j}) = j$, so $\mathsf C^\mho_j(K)$ is the \emph{top} reduced homology of the quotient $K_{-n+j}/K_{-n+j-1}$.

These homology groups can be organized into a \emph{differential complex}
\begin{eqnarray}\label{eq11.5}
\mathbf C^\mho_\ast(K) =_{\mathsf{def}} \big\{0 \to \mathsf C^\mho_n(K) \stackrel{\partial_n}{\rightarrow} \mathsf C^\mho_{n-1}(K) \stackrel{\partial_{n-1}}{\rightarrow}  \dots \stackrel{\partial_1}{\rightarrow}  \mathsf C^\mho_0(K) \to 0\big\}, \quad
\end{eqnarray}
where the differentials $\{\partial_j\}$ are the boundary homomorphisms from the long exact homology sequences of the triples $\{K_{-n+j} \supset  K_{-n+j-1} \supset K_{-n+j-2}\}_j.$ 

%For  $K = \mathcal T(v)$, the differential complex as  in (2.4), %(\ref{eq11.5AA}), 
%produced with the help of a traversally generic vector field $v \in \mathcal V^\ddagger(X)$, will be central to our investigations in this chapter and in Chapter 12.

Since the pure strata $K_{-n+j}^\circ =_{\mathsf{def}} K_{-n+j} \setminus K_{-n+j-1}$ are open orientable manifolds of dimension $j$, the $\mathcal A$-modules $\{\mathsf C^\mho_j(K)\}_j$ are \emph{free} and finitely generated, the number of generators being the number of connected components of the locus $K_{-n+j}^\circ$.
\smallskip

We call the homology groups $\mathsf H^\mho_\ast(K)$, associated with the differential complex (3.4), %(\ref{eq11.5AA}),
the $\mho$-\emph{homology} of $K$. 

In fact, the $\mho$-\emph{homology} is an ingredient of the homology spectral sequence associated with the filtration $\{K_{-j}\}_j$ and converging to the regular homology $H_\ast(K) \approx H_\ast(X)$ for any shadow  $K = F(X)$. 
Therefore the is a canonical homomorphism $A^\mho_\ast : \mathsf H^\mho_\ast(K) \to H_\ast(X)$.\smallskip

%These spectral sequence instruments and their applications to traversally  generic fields will be developed in the paper to follow. \smallskip

\noindent {\bf Definition 3.6.}
%\begin{definition}\label{def11.12} 
With any filtered $CW$-complex $K \in \mathsf {Shad}(\mathsf T, n)$ and its filtration as in (2.1), we associate the ordered collection of ranks $$\big\{c_j(K) =_{\mathsf{def}} \textup{rk}_{\mathcal A}(\mathsf C^\mho_j(K))\big\}_{0 \leq j \leq n},$$ where the groups $\{\mathsf C^\mho_j(K))\}_j$ were introduced in (3.3). %(\ref{eq11.5AA}). 

We call $c_j(K)$ \emph{the $j$-th $\mho$-complexity} of $K$.\smallskip
%Put $$\kappa_{j}(n) =_{\mathsf{def}} \max_{\{\omega|\, \mu'(\omega) = n-j\}} (\mu(\omega) - \mu'(\omega)).$$
%The \emph{weighted  $j$-complexity} of $K$ is defined by the formula $$^\sharp c^j_\mho(K) =_{\mathsf{def}}  \kappa_j(n) \cdot c^j_\mho(K).$$
\hfill $\diamondsuit$ 
%\end{definition}

%\begin{definition}\label{def11.14AA} Let $X$ be a compact connected $\mathsf{PL}$-manifold  with boundary,   $\dim(X) = n+1$. Let $$c^j_{\mathsf{shad}}(X) =_{\mathsf{def}} \min_{F \in  \mathsf {Shad}(X, \mathcal S, \mathsf T)} c^j_\mho(F(X))$$ We call  $c^j_{\mathsf{shad}}(X)$ the $j$-th \emph{shadow complexity} of $X$. \hfill $\diamondsuit$
%\end{definition}

\noindent{\bf Definition 3.7.}
%\begin{definition}\label{def11.13} 
Let $X$ be a compact connected $\mathsf{PL}$-manifold  with boundary, $\dim(X) = n+1$.
Consider the variety of maps $F: X \to K$ from the set $\mathsf {Shad}(X,  \mathsf E \Rightarrow \mathsf T)$ as in Definition 3.1.  
Each $F$ produces the sequence of $\mho$-complexities: 
$$\mathbf c(F) =_{\mathsf{def}} \big\{c_0(F(X)),\, c_1(F(X)),\, \dots ,\, c_n(F(X))\big\}$$

Consider the \emph{lexicographical minima} 
\begin{eqnarray}\label{eq11.6} 
\mathbf c^{\mathsf {shad}}(X,  \mathsf E \Rightarrow \mathsf T) =_{\mathsf{def}} \textup{lex.min}_{\{F \in \mathsf {Shad}(X, \, \mathsf E \Rightarrow \mathsf T)\}}\; \mathbf c(F(X)) 
\end{eqnarray}
We call them \emph{the lexicographic shadow complexity} of $X$. 

We denote by $c_j^{\mathsf {shad}}(X,  \mathsf E \Rightarrow \mathsf T)$ the $(j+1)$-component of the vectors  $\mathbf c^{\mathsf {shad}}(X,  \mathsf E \Rightarrow \mathsf T)$.
\hfill $\diamondsuit$
%\end{definition}
\smallskip 

\noindent{\bf Remark 3.1.}
%\begin{remark}\label{rem11.2} 
Of course, the definitions above rely on the set  $\mathsf {Shad}(X,  \mathsf E \Rightarrow \mathsf T)$ being nonempty, a nontrivial fact which requires a carefully designed poset $\mathcal S$ and a system of models $\{p_\omega: \mathsf E_\omega \to \mathsf T_\omega\}_{\omega \in \mathcal S}$ (as the ones introduced in [K2]). \hfill $\diamondsuit$
%\end{remark}
\smallskip

\noindent{\bf Remark 3.2.}
%\begin{remark}\label{rem11.3} 
If a compact connected $\mathsf{PL}$-manifold $X$ of dimension $n+1$  is globally $k$-convex (see Definition 3.2), % \ref{def11.7}), 
then for some shadow $F$ and all $j \geq k$,  we get $F(X)_{-j } = \emptyset$. Thus $\mathsf C^\mho_{n-j}(F(X)) = 0$ for all $j \geq k$, implying $c_{n-j}^{\mathsf{shad}}(X,  \mathsf E \Rightarrow \mathsf T) = 0$ for all $j \geq k$. In other words, $c_{l}^{\mathsf{shad}}(X,  \mathsf E \Rightarrow \mathsf T) = 0$ for all $l \leq n-k$.
\hfill $\diamondsuit$
%\end{remark} 
\smallskip 

Our next goal is to find some lower estimates of the complexities  $\{c_j^{\mathsf {shad}}(X,  \mathsf E \Rightarrow \mathsf T)\}_j$ in terms of the algebraic topology of $X$.
\smallskip

Let us consider a refinement $\mathcal S^\bullet_F(X)$ of the stratification $\mathcal S_F(X)$ of $X$, formed by the \emph{connected components} of the sets:
\begin{eqnarray}\label{eq11.7}
\qquad \qquad \big\{X^\circ(F, \omega) =_{\mathsf{def}} X(F, \omega) \setminus (\partial X \cap X(F, \omega)),\;  X^\partial(F, \omega) =_{\mathsf{def}} \partial X \cap X(F, \omega)\big\}_{\omega \in \mathcal S}.
\end{eqnarray}
Fig. 1 shows an example of such a stratification $\mathcal S^\bullet_F(X)$ on a surface $X$.

The double $DX$ acquires the $\tau$-equivariant stratification $\mathcal S_F(DX)$ whose pure strata are:  $$\{X^\circ(F, \omega),\,  \tau(X^\circ(F, \omega)),\,  X^\partial(\omega)\}_{\omega}.$$
and its refinement $\mathcal S^\bullet_F(DX)$, formed by the connected components of the strata from $\mathcal S_F(DX)$.

The stratifications $\mathcal S^\bullet_F(X)$ and $\mathcal S^\bullet_F(DX)$ give rise to the filtrations: 
$$X =_{\mathsf{def}} X_{-0}^F \supset X_{-1}^F \supset \dots \supset X_{-(n+1)}^F,$$
$$DX =_{\mathsf{def}} DX_{-0}^F \supset DX_{-1}^F \supset \dots \supset DX_{-(n+1)}^F$$
by the union of strata of a fixed codimension, each pure stratum being an open manifold.
 
Analogously to (\ref{eq11.4}) we consider the relative homology  and cohomology groups with coefficients in $\mathcal A$:
 \begin{eqnarray}\label{eq11.7} 
\{\mathsf C^\mho_j\big(DX, F\big) =_{\mathsf{def}} H_j\big(DX^F_{j -n - 1},\, DX^F_{j-n}; \, \mathcal A\big)\}_j,  \\ 
\{\mathsf C_\mho^j\big(DX, F\big) =_{\mathsf{def}} H^j\big(DX^F_{j -n - 1},\, DX^F_{j-n}; \, \mathcal A\big)\}_j. \nonumber
\end{eqnarray}
They can be organized into a differential complex: 
\begin{eqnarray}\label{eq11.8}
\mathbf C^\mho_\ast(DX, F) =_{\mathsf{def}} \qquad \qquad \qquad  \\ 
\big\{0 \to \mathsf C^\mho_{n+1}(DX, F) \stackrel{\partial_{n+1}}{\rightarrow} \mathsf C^\mho_n(DX, F) \stackrel{\partial_{n}}{\rightarrow}  \dots \stackrel{\partial_1}{\rightarrow}  \mathsf C^\mho_0(DX, F) \to 0\big\}, \nonumber
\end{eqnarray}
where the differentials $\{\partial_j\}$ are the boundary homomorphisms from the long exact homology sequences of triples $\{DX^F_{-j +1} \supset  DX^F_{-j} \supset DX^F_{-j -1}\}_j.$

Similarly, we introduce the dual differential complex 
\begin{eqnarray}\label{eq11.9}
\mathbf C_\mho^\ast(DX, F) =_{\mathsf{def}} \qquad \qquad \qquad \\ 
\big\{0 \leftarrow \mathsf C_\mho^{n+1}(DX, F) \stackrel{\partial^\ast_{n+1}}{\leftarrow} \mathsf C_\mho^n(DX, F) \stackrel{\partial^\ast_{n}}{\leftarrow}  \dots \stackrel{\partial^\ast_1}{\leftarrow} \mathsf C_\mho^0(DX, F) \leftarrow 0\big\} \nonumber 
\end{eqnarray}

When $\mathcal A = \Z$, since  $DX^F_{j -n - 1} \setminus DX^F_{j-n}$ is an open orientable $j$-manifold,  by \cite{Hat}, Corollary 3.28, the torsion  
$\textup{tor} \big\{H_{j-1}\big(DX^F_{j -n - 1},\, DX^F_{j-n}; \, \Z\big)\big\} =0$,
which results in the natural isomorphism  $$\mathsf C_\mho^j(DX, F) \approx \textup{Hom}_\Z(\mathsf C^\mho_j(DX, F),\, \Z).$$

In the notations to follow, we drop the dependence of the constructions on the choice of  the coefficient group $\mathcal A$, but presume that $\mathcal A$ is $\Z$ or $\R$. \smallskip

So the complexes $\mathbf C^\mho_\ast(DX, F)$ and $\mathbf C_\mho^\ast(DX, F)$ are comprised of \emph{free} $\mathcal A$-modules whose generators $\{[\s]\}$ are in 1-to-1 correspondence with the strata $\s\in \mathcal S^\bullet_F(DX)$ (see (\ref{eq11.7})).

We denote by $\mathsf Z^\mho_j(DX, F)$ the kernels  of the differentials  $\d_j$  and by $\mathsf B^\mho_j(DX, F)$ the images of the differentials  $\d_{j+1}$ from (\ref{eq11.9}).  
Next, we  form the $\mho$-\emph{homology}  $$\mathsf H^\mho_j(DX, F) =_{\mathsf{def}} \mathsf Z^\mho_j(DX, F)/ \mathsf B^\mho_j(DX, F)$$ 
of the double $DX$, associated with its stratification $\mathcal S^\bullet_F(DX)$.

Since each cycle $\zeta \in \mathsf Z^\mho_j(DX, F)$ is a common singular cycle in $DX$ (indeed, by  (\ref{eq11.8}) and (\ref{eq11.9}), the boundary of the $j$-chain $\zeta$ is mapped in the subcomplex of $DX$ of dimension $j-2$), there is a canonical homomorphism 
$A^\mho_\ast : \mathsf H^\mho_\ast(DX, F) \to  H_\ast(DX).$
%\smallskip

Since the differential complex (\ref{eq11.9}) is the dual of the differential complex (\ref{eq11.8}), the natural paring $\mathsf C^{n-k}_\mho(DX, F) \otimes \mathsf C_{n-k}^\mho(DX, F) \to \R$ produces an isomorphism $$\Phi: \mathsf C^{n-k}_\mho(DX, F)/ \mathsf B^{n-k}_\mho(DX, F) \approx  \mathsf (Z^\mho_{n-k}(DX, F))^\ast,$$
where $(Z^\mho_{n-k}(DX, F))^\ast$ denotes the dual space of $Z^\mho_{n-k}(DX, F))$. 
\smallskip

%{\bf WORK delete ????} One of the crucial questions is to understand how the differential complexes  $\mathbf C^\mho_\ast(F(X))$ in (3.4) % (\ref{eq11.9}) 
%and $\mathbf C^\mho_\ast(DX, F)$ in (3.8) % (\ref{eq11.10}) 
%vary, as $F$ runs over  the set of shadows $\mathsf {Shad}(X,  \mathsf E \Rightarrow \mathsf T)$. Such understanding would allow to treat the appropriate equivalence class of a particular differential complex $\mathbf C^\mho_\ast(F(X))$ as an invariant of $X$. We will devote another article to such investigations. 
%{\bf END WORK}
%\smallskip

The next localization construction is central to the rest of this section. 
Using the Poincar\'{e} duality $\mathcal D$ on the closed oriented manifold $DX$, for each $k \in [-1, n]$, %{\bf ????}, 
we introduce the localization transfer map
$$\mathcal L^F_{k+1}: H_{k+1}\big(DX\big) \stackrel{\approx\mathcal D}{\longrightarrow} H^{n-k}\big(DX\big) \stackrel{i^\ast}{\rightarrow}  H^{n-k}\big(DX^F_{-(k+1)}\big),$$
where $i : DX^F_{-(k+1)} \subset DX$ is the embedding.

Because the locus $DX^F_{-(k+1)}$ is a $(n-k)$-dimensional $CW$-complex, the top cohomology group $H^{n-k}(DX^F_{-(k+1)})$ can be identified with the quotient %group 
$\mathsf C^{n-k}_\mho(DX, F)/ \mathsf B^{n-k}_\mho(DX, F)$ from the complex in  (\ref{eq11.9}). This fact follows from the diagram chase in the two long exact cohomology sequences of the pairs $\big(DX^F_{-(k+1)}, DX^F_{-(k+2)}\big), \; \big(DX^F_{-(k+2)}, DX^F_{-(k+3)}\big).$ The first pair gives rise to the fragment of the exact sequence 
$$\to H^{n-k-1}(DX^F_{-(k+2)}) \stackrel{\delta}{\rightarrow}  \mathsf C^{n-k}_\mho(DX, F) \stackrel{j^\ast}{\rightarrow} H^{n-k}(DX^F_{-(k+1)}) \to 0,$$
while the second pair to the fragment
$$ \to \mathsf C^{n-k -1}_\mho(DX, F)  \stackrel{J^\ast}{\rightarrow} H^{n-k-1}(DX^F_{-(k+2)}) \to 0,$$ 
so that the composition $J^\ast \circ \delta$ coincides  with the boundary map $$\delta^{n-k-1} : \mathsf C^{n-k -1}_\mho(DX, F) \to \mathsf C^{n-k}_\mho(DX, F)$$ from the differential complex $\mathsf C^\ast_\mho(DX, F)$.  Therefore $$H^{n-k}(DX^F_{-(k+1)}) \approx \mathsf C^{n-k}_\mho(DX, F)/ \mathsf B^{n-k}_\mho(DX, F).$$
Similar diagram chase in homology leads to an isomorphism 
$$H_{n-k}(DX^F_{-(k+1)}) \approx \mathsf Z_{n-k}^\mho(DX, F).$$

As a result, for each shadow $F$, we get the transfer homomorphism 
\begin{eqnarray}\label{eq11.11}
\qquad \qquad\mathcal L_{k+1}^{F, \mho}: H_{k+1}(DX) \stackrel{\approx\mathcal D}{\longrightarrow} H^{n-k}(DX) \stackrel{i^\ast}{\rightarrow} %\mathsf Z^{n-k}_\mho(DX, F).
\mathsf C^{n-k}_\mho(DX, F)/ \mathsf B^{n-k}_\mho(DX, F).
\end{eqnarray}

Our next goal is to introduce a norm $|[\sim ]|_\mho$ in the target space $\mathsf C^{n-k}_\mho(DX, F)/ \mathsf B^{n-k}_\mho(DX, F)$ %$\mathsf C^{n-k}_\mho(DX, F)\supset \mathsf Z^{n-k}_\mho(DX, F)$
of the homomorphism $\mathcal L_{k+1}^{F, \mho}$ and to show that 
\begin{eqnarray}%\label{eq11.A}
\|h\|_\D \leq const(n) \cdot |[\mathcal L_{k+1}^{F, \mho}(h)]|_\mho \nonumber
\end{eqnarray}  
for any $h \in H_{k+1}(DX)$ and some \emph{universal} constant $const(n) > 0$ which depends only on the list of model maps $\{\mathsf E_\omega \to \mathsf T_\omega\}_\omega$ as in the last bullet of Definition 3.1. 

Consider the $l_1$-norm $\|\sim\|_{l_1}$ on the space $\mathsf C^{n-k}_\mho(DX, F)$ in the unitary basis $\{[\s^\ast]\}$, labeled by the strata $\s \in \mathcal S^\bullet_F(DX)$ of dimension $n - k$, and dual to the preferred basis $\{[\s]\}$ of $\mathsf C_{n-k}^\mho(DX, F)$. Then the norm $ |[\sim]|_\mho$ on $\mathsf C^{n-k}_\mho(DX, F)/ \mathsf B^{n-k}_\mho(DX, F)$ is, by definition, the quotient norm induced by the norm $\|\sim\|_{l_1}$ on the space $\mathsf C^{n-k}_\mho(DX, F)$.

We may regard the maximal proportion $$\sup_{\{h \in H_{k+1}(DX), \, \|h\|_\D \neq 0\}}\, \big\{ |[\mathcal L_{k+1}^{F, \mho}(h)]|_\mho/\|h\|_\D\big\}$$ as the norm of the operator $\mathcal L_{k+1}^{F, \mho}$. So the inequality above (see also (3.11)) testifies that these norms admit a positive lower bound,  universal for all $X$, $\dim(X) = n+1$,  and $F$.
\smallskip

As the reader examines the hypotheses of the next theorem, it is helpful to keep in mind the following simple example. Let $X$ be a 2-dimensional ball with two holes. Then $DX$ is a closed surface of  genus 2. %, and the quotient $X/\d X$ can be viewed as a sphere with three of its points being identified. Thus the fundamental group $\Pi = \pi_1(DX)$ has the presentation $\{a, b, c, d | \; aba^{-1}b^{-1}cdc^{-1}d^{-1} = 1\},$ while the fundamental group $\pi = \pi_1(X)$ is a free group with two generators. 
Note that the fundamental group of each component of $\d X$ is abelian. So its image in $\pi_1(DX)$  is an amenable group.
% Collapsing the second copy of $X$ in $DX$ produces a map $\b: DX \to X/\d X$ so that $\b_\ast(a) = e,\, \b_\ast(b) = 1,\, \b_\ast(c) = f,\, \b_\ast(d) = 1$. {\bf WORK}\smallskip
%%%% WORK

\begin{theorem}[{\bf The $\mathbf 1^{st}$ Amenable Localization of the Poincar\'{e} Duality}]\label{th11.3} \hfill\break
Let $X$ be a compact connected $(n+1)$-dimensional $\mathsf{PL}$-manifold with a nonempty boundary. 

Assume that for each connected component of the boundary $\d X$, the image of its fundamental group in $\pi_1(DX)$ is an amenable group\footnote{Evidently, if for each connected component of the boundary $\d X$, the image of its fundamental group in $\pi_1(X)$ is amenable, then it is automatically amenable in $\pi_1(DX)$.}.

Then, there exists an universal constant $\Theta > 0$ such that, for any shadow  $F \in \hfill\break \mathsf {Shad}(X,  \mathsf E \Rightarrow \mathsf T)$, the space $\mathsf C^{n-k}_\mho(DX, F)/ \mathsf B^{n-k}_\mho(DX, F)$ admits a norm $|[\; \; ]|_\mho$ so that 
\begin{eqnarray}\label{eq11.12}
\|h\|_\D \leq \Theta \cdot |[\mathcal L_{k+1}^{F, \mho}(h)]|_\mho 
\end{eqnarray}  
for all $h \in H_{k+1}(DX)$. Here  the Poincar\'{e} duality localizing operator $\mathcal L_{k+1}^{F, \mho}$ is introduced in (\ref{eq11.11}), and the constant $\Theta$ depends only on $n$, the poset $\mathcal S$, and the list of model maps $\{\mathsf E_\omega \to \mathsf T_\omega\}_{\omega \in \mathcal S}$ in the way that is described in formula (\ref{eq11.16a}). 
\end{theorem}

\begin{proof} Let $U$ be a \emph{regular} neighborhood of the set $DX^F_{-(k+1)}$ in $DX$.  So the $\mathsf{PL}$-manifold $U$ has a homotopy type of a $(n - k)$-dimensional $CW$-complex $DX^F_{-(k+1)}$.

Let $\Pi =_{\mathsf{def}} \pi_1(DX)$. In order to apply Localization Lemma \ref{lem11.1} to the neighborhood $U$, the homology class $h \in H_{k+1}(DX; \Z)$, and the classifying map  $\b: DX \to K(\Pi, 1)$, we need to check that, for each stratum $\s \in \mathcal S^\bullet_F(DX)$, the subgroup $\b_*\pi_1(\s)$ of $\Pi$ is an amenable group.  This is true if $\s \subseteq \d X$ since for each connected component of $\d X$, the image of its fundamental group in $\Pi$ is  amenable, and every subgroup of an amenable group is amenable.  Otherwise,  $\s$ is contained in some locus $X^\circ(\omega)$ or in some locus $\tau(X^\circ(\omega))$. 

By the fourth bullet of Definition 3.1, this preimage $X^\circ(F, \omega)$ is a trivial oriented bundle $F: X^\circ(F, \omega) \to K(\omega)$ whose fiber is a disjointed union of open intervals.  Therefore \hfill\break $F: \s \to F(\s) \subset K(\omega)$ is a trivial fibration whose fiber is an open interval. As a result, $F_\ast: \pi_1(\s) \to \pi_1(F(\s))$ is an isomorphism.

On the other hand, the covering  $F: X^\partial(F, \omega) \to K(\omega)$ is trivial and therefore admits a section $\rho: K(\omega) \to  X^\partial(F, \omega)$ such that its image intersects with the closure $\textup{cl}(\s)$ of $\s$ in $X$. Put $\bar \s =_{\mathsf{def}} \s \cup \rho(F(\s))$. The map $F: \bar \s \to F(\s)$ is a trivial  fibration with a semi-open interval for the fiber. Therefore the imbedding  $j: \s  \subset \bar \s$ is a homotopy equivalence, and so is the obvious imbedding $q: \rho(F(\s)) \subset \bar \s$. Hence $j_\ast: \pi_1(\s) \to \pi_1(\bar \s)$ and $q_\ast: \pi_1(\rho(F(\s))) \to \pi_1(\bar \s)$ are isomorphisms. Since $\rho(F(\s)) \subset \d X$ and, for every choice of the base point $x_\star \in \d X$, the $\b_\ast$-image of each fundamental group $\pi_1(\d X, x_\star)$ in $\Pi$ is amenable, so is the  $\b_\ast$-image of  $\pi_1(\rho(F(\s)))$---a subgroup of an amenable group  is amenable. By the previous arguments, the $\b_\ast$-image of  $\pi_1(\s)$  is amenable.

Now, applying localization Lemma 3.1 to the classifying map $\b: DX \to K(\Pi, 1)$, gives the inequality:  
 \begin{eqnarray}\label{eq11.13}
 \| h \|_\D \leq \| h_U \|_\D^{\mathcal S^\bullet_F(U)},
 \end{eqnarray}
 where $h_U$ denotes the restriction of the absolute homology class $h$ to $(U, \d U)$. Indeed by \cite{Gr}, $\|\b_\ast(h)\|_\D = \|h\|_\D$ since, by its construction, $\b$ induces tan isomorphism of the fundamental groups of $DX$ and $K(\Pi, 1)$. \smallskip
 
 \begin{figure}[ht]
\centerline{\includegraphics[height=3in,width=4in]{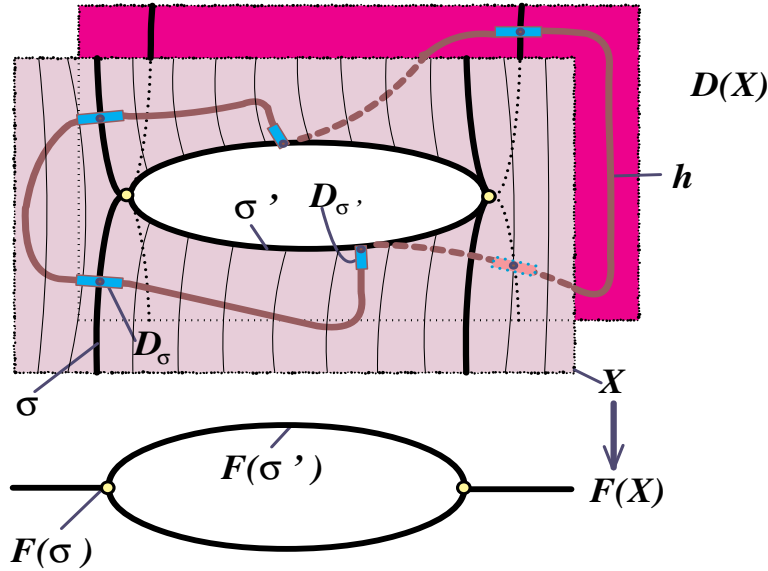}} 
\bigskip
\caption{The cycle $h \in H_{k+1}(DX)$, its transversal intersections with the strata $\{\s \subset DX^F_{-(k+1)}\}$,  and the normal disks $\{D_\s^{k+1}\}$ which represent the localization $h_U$.}
\end{figure}

 For each $(n-k)$-dimensional stratum $\s \in \mathcal S^\bullet_F(DX)$, consider an oriented disk $$(D_\s^{k+1}, \d D_\s^{k+1}) \subset (U, \d U),$$ normal to the open manifold $\s\subset U$ at its typical point. Taking a smaller regular neighborhood $U$ of the subcomplex $DX^F_{-(k+1)} \subset DX$ if necessarily, we can arrange for $\{D_\s^{k+1}\}_\s$ to be disjointed, so that each disk  $D_\s^{k+1}$ hits its stratum $\s$ transversally at a singleton and misses the rest of the strata. Note that the relative integral homology classes $\{[D_\s^{k+1}, \d D_\s^{k+1}]\}_\s$ may be dependent in $H_{k+1}(U, \d U)$. 

We claim that any element $h_U \in H_{k+1}(U, \d U)$ can be written as a linear combination of relative cycles $\{[D_\s^{k+1}, \d D_\s^{k+1}] \in H_{k+1}(U, \d U)\}_\s$: 
\begin{eqnarray}\label{eq11.14}
h_U =  \sum_{\s \in\mathcal S^\bullet_F(DX), \, \dim \s = n-k} r_\s \cdot [D_\s^{k+1}, \d D_\s^{k+1}].
\end{eqnarray}
In fact, $H_{k+1}(U, \d U)$ can be recovered from the free $\Z$-module, generated by the elements $\{[D_\s^{k+1}, \d D_\s^{k+1}]\}_\s$, by factoring the module by the appropriate relations. %{\bf WORK} 

To justify the presentation in (\ref{eq11.14}), we notice that, since  $DX^F_{-(k+1)}$ is a deformation retract of $U$, any $(n-k)$-dimensional homology class  in $U$ is  represented by a cycle $z$ in $DX^F_{-(k+1)}$, a combination $\sum_\s n_\s \cdot \s$ of the top strata of the $(n-k)$-dimensional $CW$-complex $DX^F_{-(k+1)}$. The algebraic intersection of $z \circ D_\s^{k+1} = n_\s$, since $\s \circ   D_\s^{k+1}  =1$ and $\s' \circ   D_\s^{k+1}  = 0$ for any $\s' \neq \s$ by the very choice of the normal disks $ D_\s^{k+1}$'s. For the $\R$-coefficients, by the Poincar\'{e} duality, any $(k+1)$-dimensional homology class $\tilde h \in H_{k+1}(U, \d U)$ is determined by its algebraic intersections with the cycles $z \in H_{n-k}(U) \approx H_{n-k}(DX^F_{-(k+1)})$.  Therefore any $\tilde h$ is a linear combination of $\{[D_\s^{k+1}, \d D_\s^{k+1}]\}_\s$. 

To simplify notations, put $[D_\s^{k+1}] =_{\mathsf{def}} [D_\s^{k+1}, \d D_\s^{k+1}]$.
\smallskip

We introduce the norm of $h_U$ by the formula
\begin{eqnarray}\label{eq11.15}
|[h_U]| =_{\mathsf{def}} \inf_{\{\text{representations of}\, h_U\}} \Big\{\sum_{\s \in\mathcal S^\bullet_F(DX), \, \dim \s = n-k} |r_\s|\Big\}, 
\end{eqnarray}
the minimum being taken over all representations of $h_U$ as in (\ref{eq11.14}). \smallskip 

Applying (\ref{eq11.13}) to any presentation of $h_U$ as in (\ref{eq11.14}), we get
 \begin{eqnarray}\label{eq11.16}
 \| h \|_\D \leq \sum_{\s \in\mathcal S^\bullet_F(DX), \, \dim \s = n-k} |r_\s| \cdot \big\| [D_\s^{k+1}]\big \|_\D^{\mathcal S^\bullet_F(U)}.
 \end{eqnarray}

Thus, 
 \begin{eqnarray}\label{eq11.16a}
 \| h \|_\D \leq \Theta \cdot \sum_{\s \in\mathcal S^\bullet_F(DX), \, \dim \s = n-k} |r_\s|, \nonumber
 \end{eqnarray}
where 
\begin{eqnarray}\label{eq11.17}
 \Theta  =_{\mathsf{def}} \max_{\s \in\mathcal S^\bullet_F(DX), \, \dim \s = n-k} \Big\{\big\| [D_\s^{k+1}]\big \|_\D^{\mathcal S^\bullet_F(U)}\Big\}.
\end{eqnarray}
Here we stratify the normal disk $D_\s^{k+1}$ by intersecting it with the $\mathcal S^\bullet_F(DX)$-stratification in the ambient space $DX$.  Employing Definition 3.1, % \ref{def11.6}, 
the normal to $\s$ disk $D_\s^{k+1}$ can be viewed as a subspace of the model double space $D\mathsf E_{\omega_\s}$, stratified with the help of the model map $p_{\omega_\s}: \mathsf E_{\omega_\s} \to \mathsf T_{\omega_\s}$. In $D\mathsf E_{\omega_\s}$, the two copies of the space $\mathsf E_{\omega_\s}$, given by $\wp(u, \vec x) \leq 0$ (see (\ref{eq2.4})), are attached along the locus $\wp(u, \vec x) = 0$. The normal disk acquires its stratification from the ambient space $D\mathsf E_{\omega_\s}$. Hence, as a stratified topological space, $D_\s^{k+1}$  depends only on the position of the stratum $\s$ in the canonical stratification of $D\mathsf E_{\omega_\s}$. That position is determined by the appropriate combinatorial data, provided by formula (2.4). As a result, the stratified simplicial norms $\big\| [D_\s^{k+1}]\big \|_\D^{\mathcal S^\bullet_F(U)}$ take a \emph{finite set of values}, which depend only on the list of model maps $\{p_{\omega}: \mathsf E_{\omega} \to \mathsf T_{\omega}\}_\omega$.
%the combinatorial type $\omega_\s \in  \mathcal S$ of the stratum $\s$ and on  whether the stratum belongs to the cruder strata $X^\circ(F, \omega_\s)$ or to $X^\d(F, \omega_\s)$ {\bf CHECK  WELL !!!!} 
So the constant $ \Theta > 0$ in (\ref{eq11.17}) is universal for any $X$ and its  shadow $F \in  \mathsf {Shad}(X,  \mathsf E \Rightarrow \mathsf T)$.%; in fact, $\Theta$ depends only on the list of the model spaces from the last bullet of Definition \ref{def11.6}.
\smallskip 

Therefore, in line with formula (\ref{eq11.13}) and the definition in (\ref{eq11.15}), we get
\begin{eqnarray}\label{eq11.18}
\| h \|_\D \leq  \Theta \cdot |[h_U]|
\end{eqnarray}
where  $\Theta > 0$ is universal for all $X$, $F$, and $h$. %{\bf WORK}
\smallskip

Next, we are going to reinterpret the norm $|[h_U]|$ in terms of the differential complexes $ \mathbf C_\mho^\ast(DX, F)$ and $\mathbf C^\mho_\ast(DX, F)$ to make it ``more computable".

Let  $\{[\s]^\ast\}_\s$ be the basis of $\mathsf C_\mho^{n-k}(DX, F)$, dual to the basis $\{[\s]\}_\s$ in $\mathsf C^\mho_{n-k}(DX, F)$. 

The $|[\sim]|^\mho$-norm of a class $\eta^\ast \in \mathsf C^{n-k}_\mho(DX, F)/ \mathsf B^{n-k}_\mho(DX, F)$ is the quotient norm,  induced by the $l_1$-norm on $\mathsf C^{n-k}_\mho(DX, F)$; it is defined by the formula (which resembles the formula in (\ref{eq11.15})) 
$$|[\eta^\ast]|^\mho =_{\mathsf{def}} \inf_{\{\zeta^\ast\, \equiv \, \eta^\ast \mod\, \mathsf B^{n-k}_\mho(DX, F)\}} \Big\{\sum_{\s \in\mathcal S^\bullet_F(DX), \, \dim \s = j} |r_\s| \Big\},$$
where $$\zeta^\ast = \sum_{\s \in\mathcal S^\bullet_F(DX), \, \dim \s = j} r_\s\cdot [\s]^\ast.
\footnote{The unit balls in the $l_1$-norms on the spaces  $\mathsf C_\mho^j(DX, F)$ are convex closures of the vectors  $\{\pm[\s]^\ast\}_\s$, where $\s$ are strata of  dimension $j$.}$$  

%, both spaces being equipped with the unitary bases $\{[\s]\}_\s$ and $\{[\s]^\ast\}_\s$ whose vectors are in $1$-to-$1$ correspondence with the elements of the sets $j$-dimensional strata $\s \in\mathcal S^\bullet_F(DX)$.

%Recall that  the modules $\mathbf Z_\mho^j(DX, F)$ and  $\mathbf Z_\mho^j(DX, F)$ are dual.  

\smallskip 

Recall that in (\ref{eq11.14}) we have considered an epimorphism $A: \mathsf C_{k+1}^\dagger \to H_{k+1}(U, \d U)$, where $\mathsf C_{k+1}^\dagger $ denotes the free module over $\R$ (or over  $\Z$), generated by the normal relative disks $\{D_\s^{k+1}\}_\s$. Put $\mathsf R_{k+1}^\dagger =_{\mathsf{def}} \ker(A)$, so that $$\mathsf C_{k+1}^\dagger/\mathsf R_{k+1}^\dagger \approx H_{k+1}(U, \d U).$$

On the other hand, the Poincar\'{e} duality $\mathcal D_U$ produces an isomorphism
$$B: H_{k+1}(U, \d U) \stackrel{\approx\mathcal D_U}{\longrightarrow} H^{n-k}(U) \stackrel{\approx}{\longrightarrow} \mathsf C^{n-k}_\mho(DX, F)/ \mathsf B^{n-k}_\mho(DX, F).$$ The composition $B \circ A$ takes each generator $D_\s^{k+1} \in  \mathsf C_{k+1}^\dagger$ to the class of $[\s]^\ast$ and identifies the space of relations $\mathsf R_{k+1}^\dagger$ with the space $\mathsf B^{n-k}_\mho(DX, F)$.

% KEEP !!!   Recall that we have established the isomorphisms  $$\mathsf C_\mho^{n-k}(DX, F)/\mathsf B_\mho^{n-k}(DX, F) \approx H^{n-k}(DX^F_{-(k+1)}),$$ $$\mathsf Z^\mho_{n-k}(DX) \approx H_{n-k}(DX^F_{-(k+1)}).$$ Since $$H^{n-k}(DX^F_{-(k+1)}) \approx \textup{Hom}_\Z\big(H_{n-k}(DX^F_{-(k+1)}), \Z\big)$$ due to the fact that $DX^F_{-(k+1)}$ is $(n-k)$-dimensional,  the duality follows.

%{\bf WORK}
%\bigskip

Let $\mathcal D_U(h_U)$ denotes the Poincar\'{e} dual in $U$ of the relative homology class $h_U$. Examining the definitions of the norm in (\ref{eq11.15}) and of the norm $|[\sim]|^\mho$ and tracing the nature of the Poincar\'{e} duality (in terms of the intersections of relative and absolute cycles of complementary dimensions in $U$), we see that the norm $|[h_U]|$ in (\ref{eq11.15}) coincides with the norm $|[\mathcal D_U(h_U)]|^\mho$, where  
$$\mathcal D_U(h_U) \in H^{n-k}(DX^F_{-(k+1)}) \approx  \mathsf C^{n-k}_\mho(DX, F)/ \mathsf B^{n-k}_\mho(DX, F).$$

Therefore, combining this observation with (\ref{eq11.18}), we get the desired inequality: 
\begin{eqnarray}\label{eq11.19}
\|h\|_\D \leq \Theta \cdot |[\mathcal L_{k+1}^{F, \mho}(h)]|_\mho.
\end{eqnarray}

\quad 
\end{proof} 

Theorem 3.1 has a ``more geometric" interpretation (see Fig. 3). %in terms of counting the intersections of $(k+1)$-cycles in $DX$ with the locus $DX^F_{-(k+1)}$ of the complementary dimension. 
%Fig. 3 illustrates the idea of the arguments below.

\begin{corollary}\label{cor11.2} Let a homology class $h \in H_{k+1}(DX; \Z)$ be realized my a singular pseudo-manifold\footnote{a compact simplicial complex whose singular set has codimension $2$ at least} $f: M \to DX$, $\dim(M) = k+1$.  

Under the hypotheses of Theorem \ref{th11.3a}, 
the number of intersections of $f(M)$ with the locus $DX^F_{-(k+1)}$ is greater than or equal to $\Theta^{-1} \cdot \|h\|_\D$, provided that $f$ is in general position with the subcomplex $DX^F_{-(k+1)} \subset DX$.
\end{corollary}

\begin{proof} Recall that any integral homology class $h \in H_{k+1}(DX; \Z)$ can be realized by a pseudo-manifold  $f: M \to DX$ (\cite{Hat}, pages 108-109).  %$(2m+1)$-multiple of a singular oriented manifold $f: M \to DX$, where the integer $m$ depends only on $X$ and $h$ \cite{CF}. 

By a small perturbation of $f$, we can assume that $f(M)$ intersects transversally only with the pure $(n-k)$-dimensional strata in $DX^F_{-(k+1)}$ from the poset $\mathcal S^\bullet_F(DX)$. 

Consider the transversal intersections from the set $f(M) \cap DX^F_{-(k+1)}$. Then, for a sufficiently narrow regular neighborhood $U$ of $DX^F_{-(k+1)}$, the localized class $h_U$ has a representation 
$$h_U =  \sum_{x_\s \in f(M) \cap DX^F_{-(k+1)}} r_\s \cdot [D_{x_\s}^{k+1}],$$
where the relative disk  $(D_{x_\s}^{k+1}, \d D_{x_\s}^{k+1}) \subset (U, \d U) \cap f(M)$ and $r_\s = \pm 1$. By (\ref{eq11.19}), the norm $$|[ h_U ]| \leq \sum_{x_\s \in f(M) \cap DX^F_{-(k+1)}} 1 \; = \; \#\Big( f(M) \cap DX^F_{-(k+1)}\Big).$$ 
Again, by  (\ref{eq11.19}), we get  $\| h \|_\D \leq \Theta \cdot \#\Big( f(M) \cap DX^F_{-(k+1)} \Big)$. Thus 
\begin{eqnarray}\label{eq11.20}
\#\Big( f(M) \cap DX^F_{-(k+1)} \Big) \geq \Theta^{-1} \cdot \| h \|_\D,
\end{eqnarray}
where the positive constant $\Theta^{-1}$ depends only on $n$, the poset $\mathcal S$, and the list of model maps $\mathsf E \Rightarrow \mathsf T$ in the way that is described in formula (\ref{eq11.17}). 
%\smallskip
\end{proof}
\smallskip

%%%%%%%%
%%%%%%%
%{\bf WORK} 
In order to get similar results about \emph{absolute} homology classes $h \in H_{k+1}(X)$, we consider the $F$-induced filtration
$$X_{-0}^{F\circ} \supset X_{-2}^{F\circ} \supset \dots \supset X_{-n}^{F\circ}$$
of $\mathsf{int}(X)$ by the connected components  of the strata $\{X^\circ(F, \omega)\}$ as in (3.6) and the relative homology/cohomology groups 
\begin{eqnarray}\label{eq11.8a} 
\{\mathsf C^\mho_j\big(X, F^\circ \big) =_{\mathsf{def}} H_j\big(X^{F \bullet}_{j -n - 1},\, X^{F \bullet }_{j-n} \cup (X^{F \bullet}_{j -n - 1}  \cap \d X); \, \mathcal A\big)\}_j, \nonumber \\  
\{\mathsf C_\mho^j\big(X, F^\circ\big) =_{\mathsf{def}} H^j\big(X^{F \bullet}_{j -n - 1},\, X^{F \bullet }_{j-n} \cup (X^{F \bullet}_{j -n - 1}  \cap \d X); \, \mathcal A\big\}_j . %\nonumber 
\end{eqnarray}
They are designed to emphasize the role of the connected components of $\{X^\circ(F, \omega)\}$ from the stratification  $\mathcal S_F^\circ(X)$ of $\mathsf{int}(X)$. 

%By the excision axiom, these groups are isomorphic to 
%\begin{eqnarray}\label{eq11.8b} 
%\{\mathsf C^\mho_j\big(X, F^\circ \big) =  H_j\big(X^{F \bullet}_{j -n - 1},\, X^{F \bullet }_{j-n} \cup (X^{F \bullet}_{j -n - 1}  \cap \d X); \, \mathcal A\big)\}_j, \nonumber \\ 
%\{\mathsf C_\mho^j\big(X, F^\circ\big) =  H^j\big(X^{F \bullet}_{j -n - 1},\, X^{F \bullet }_{j-n} \cup (X^{F \bullet}_{j -n - 1}  \cap \d X); \, \mathcal A\big)\}_j,   
%\end{eqnarray}
%respectively. %Therefore, the $\mathcal A$-modules $\mathsf C^\mho_j\big(X, F^\circ \big)$ and $\mathsf C_\mho^j\big(X, F^\circ \big)$ are free and their bases are in 1-to-1 correspondence with the connected components of $\{X^\circ(F, \omega)\}$.

Using the excision property, the long exact sequences of the triples $$\{X^{F \bullet}_{-j +1} \cup \d X \supset  X^{F \bullet}_{-j} \cup \d X \supset X^{F \bullet}_{-j -1} \cup \d X\}_j$$ help to organize the groups in  (\ref{eq11.8a}) into differential complexes  $\mathbf C^\mho_\ast(X, F^\circ)$ and $\mathbf C_\mho^\ast(X, F^\circ)$, similar to the ones in (\ref{eq11.8}) and (\ref{eq11.9}).   
\smallskip

%%%%
%%%%

Employing the Poincar\'{e} duality $\mathcal D$ again, we introduce the localization homomorphism
$$\mathcal M^F_{k+1}: H_{k+1}(X) \stackrel{\approx\mathcal D}{\longrightarrow} H^{n-k}(X, \d X) \stackrel{i^\ast}{\rightarrow}  H^{n-k}\big(X^{F \bullet}_{-(k+1)},  X^{F \bullet}_{-(k+1)} \cap \d X\big)$$
whose target can be identified with the quotient $\mathsf C^{n-k}_\mho(X, F^\circ)/ \mathsf B^{n-k}_\mho(X, F^\circ).$ This isomorphism can be validated in a way that is similar to the one that led  to (\ref{eq11.11}). Indeed, put $A =_{\mathsf{def}} X_{-(k+1)}^{F\bullet}$, $B =_{\mathsf{def}} (X_{-(k+1)}^{F\bullet} \cap \d X) \cup X_{-(k+2)}^{F\bullet}$, and $C =_{\mathsf{def}}  (X_{-(k+1)}^{F\bullet} \cap \d X)$.  Consider the fragment $J: H^{n-k}(A, B) \to H^{n-k}(A, C)$ of the long exact sequence of the triple $A \supset B \supset C$. 
By definition, $H^{n-k}(A, C)$ is the target of $\mathcal M^F_{k+1}$, and $H^{n-k}(A, B) = \mathsf C_\mho^{n-k}(X, F^\circ)$. Since $\dim(B) < n-k$,  $H^{n-k}(B, C) = 0$; so $J$ is an epimorphism. 

Similar arguments, which involve a different triple of spaces, imply that $\ker J$ can be identified with $\textup{im} \big(\d^\ast:  \mathsf C_\mho^{n-k-1}(X, F^\circ) \to \mathsf C_\mho^{n-k}(X, F^\circ)\big)$.

%The relative cohomology group $$H^{n-k}\big(X^{F \bullet}_{-(k+1)} \cup \d X, X^{F \bullet}_{-(k+2)} \cup \d X\big) \stackrel{excision}{\approx} H^{n-k}\big(X^{F \bullet}_{-(k+1)}, X^{F \bullet}_{-(k+2)} \cup (X^{F \bullet}_{-(k+1)} \cap \d X)\big)$$ is isomorphic to the quotient $\mathsf C^{n-k}_\mho(X, F^\circ)/ \mathsf B^{n-k}_\mho(X, F^\circ).$ {\bf CHECK WELL !!!!}

As a result, for each shadow $F$, we get a localization of the  Poincar\'{e} duality: 
\begin{eqnarray}\label{eq11.11a}
\qquad \qquad\mathcal M_{k+1}^{F, \mho}: H_{k+1}(X) \stackrel{\approx\mathcal D}{\longrightarrow} H^{n-k}(X, \d X) \stackrel{I^\ast}{\rightarrow} %\mathsf Z^{n-k}_\mho(D(X), F).
\mathsf C^{n-k}_\mho(X, F^\circ)/ \mathsf B^{n-k}_\mho(X, F^\circ). %\nonumber
\end{eqnarray}

As before, the target space of  $\mathcal M_{k+1}^{F, \mho}$ admits the $|[\sim]|_\mho$ quotient norm.  It is induced by the $l_1$-norm on the based space $\mathsf C^{n-k}_\mho(X, F^\circ)$, the base being indexed by the connected components of $\{X^\circ(F, \omega)\}_\omega$. \smallskip

These considerations lead to a twin of Theorem \ref{th11.3}.
%%%%%
\begin{theorem}[{\bf The $\mathbf 2^{nd}$ Amenable Localization of the Poincar\'{e} Duality}]\label{th11.3a} \hfill\break
Let $X$ be a compact connected $(n+1)$-dimensional $\mathsf{PL}$-manifold with a nonempty boundary.  

Assume that for each connected component of the boundary $\d X$, the image of its fundamental group in $\pi_1(X)$ is an amenable group.

Then, there exists an universal constant $\Theta^\circ > 0$ such that, for every shadow  $F \in \hfill\break \mathsf {Shad}(X,  \mathsf E \Rightarrow \mathsf T)$, the space $\mathsf C^{n-k}_\mho(X, F^\circ)/ \mathsf B^{n-k}_\mho(X, F^\circ)$ admits a norm $|[\; \; ]|_\mho$ so that 
\begin{eqnarray}\label{eq11.12}
\|h\|_\D \leq \Theta^\circ \cdot |[\mathcal M_{k+1}^{F, \mho}(h)]|_\mho 
\end{eqnarray}  
for all $h \in H_{k+1}(X)$. Here  the Poincar\'{e} duality localizing operator $\mathcal M_{k+1}^{F, \mho}$ is introduced in  (\ref{eq11.11a}), 
and the constant $\Theta^\circ > 0$ depends only on $n$, the poset $\mathcal S$, and the list of model maps $\{\mathsf E_\omega \to \mathsf T_\omega\}_{\omega \in \mathcal S}$ in the way that is described in formula (\ref{eq11.17a}). 
\end{theorem}

\begin{proof}  When  $h$  is represented  by a singular cycle whose image is contained in $\mathsf{int}(X)$,  it  interacts only with the strata from $\mathcal S_F^\circ(X)$ (and misses the strata from $\mathcal S_F^\bullet(\d X)$). Therefore,  we can pick a regular neighborhood $U$ of $X^{F \circ}_{-(k+1)}$ in $X$ so small that the localized class $h_U \in H_{k+1}(U, \d U)$ is a linear combination of a disjoint union of disks $\{(D_{x_\s}^{k+1}, \d D_{x_\s}^{k+1})\}$, normal to the strata $\s$ in $X^{F \circ}_{-(k+1)}$.

We apply the Localization Lemma \ref{lem11.1} to the image of the class $h$ under the classifying map $\a: X \to K(\pi_1(X), 1)$. By the argument that we used to prove Theorem \ref{th11.3}, the images of the fundamental groups of the  strata from $\mathcal S_F^\circ(X)$ are amenable in $\pi_1(X)$ .

Since only the disks  $\{D_{x_\s}^{k+1}\}_{\s \in \mathcal S_F^\circ(X)}$ participate in the localization of $h$ to $U$, the constant $\Theta^{-1}$ in the estimate (\ref{eq11.20}) can be improved by replacing $\Theta$ from formula (\ref{eq11.17}) by a smaller positive constant
\begin{eqnarray}\label{eq11.17a}
\Theta^\circ  =_{\mathsf{def}} \max_{\s \in\mathcal S^\circ_F(X), \, \dim \s = n-k} \Big\{\big\| [D_\s^{k+1}]\big \|_\D^{\mathcal S^\bullet_F(U)}\Big\}
\end{eqnarray}
which has a smaller universal upper bound, also depending  on $\mathcal S$, $n$, and $k$ only. 

The rest of the arguments are identical with the ones from the proof of Theorem \ref{th11.3}.
\end{proof}

Retracing the proof of Theorem \ref{th11.3a}, we get the following

%{\bf END WORK}
%%%%%%%%
%%%%%%%%%
%\bigskip

\begin{corollary}\label{cor11.3} Under the amenability hypotheses of Theorem \ref{th11.3a},  if $h \in H_{k+1}(X; \Z)$ is represented by a singular pseudo-manifold $f: M \to X$, then $(\Theta^\circ)^{-1} \cdot \|h\|_\D$ gives a lower bound of the number of transversal intersections of the absolute cycle $f(M)$ with the locus $X^{F \circ}_{-(k+1)} =_{\mathsf{def}} \mathsf{int}(X) \cap X^F_{-(k+1)}$.   \hfill $\diamondsuit$
%Dropping the amenability hypotheses of Theorem \ref{th11.3}, if $h \in H_{k+1}(X, \d X; \Z)$ is represented by a singular \emph{closed} manifold $f: M \to X$\footnote{and thus $h$ belongs to the image of the natural homomorphism $j_\ast: H_{k+1}(X; \Z) \to H_{k+1}(X, \d X; \Z)$}, then $\Theta^{-1} \cdot \|h\|_\D$ gives a lower bound of the number of transversal intersections of the absolute cycle $f(M)$ with the locus $X^{F \circ}_{-(k+1)} =_{\mathsf{def}} \mathsf{int}(X) \cap X^F_{-(k+1)}$.  
%%Let a homology class $h \in H_{k+1}(X, \d X; \Z)$ be realized my a singular compact oriented $\mathsf{PL}$-manifold $f: (M, \d M) \to (X, \d X)$, $\dim(M) = k+1$.  Then, the cardinality  of the intersections $f(M) \cap X^F_{-(k+1)}$ is greater than or equal to $\Theta^{-1} \cdot \|h\|_\D$, provided that $f$ is in general position with the subcomplex $X^F_{-(k+1)} \subset X$. \hfill $\diamondsuit$
\end{corollary}

For a pair of topological spaces  $Z \supset W$ and an integer $j \geq 0$, consider the subspace/subgroup  $\mathcal K^{\D = 0}_j(Z, W) \subset H_j(Z, W),$ formed by the homology classes  $h \in H_j(Z, W)$ whose simplicial semi-norm  $\| h\|_\D  = 0$. Then $\| \sim \|_\D$ becomes a norm on the quotient space $$H^{\mathbf \D}_j(Z, W) =_{\mathsf{def}} H_j(Z, W)/\mathcal K^{\D = 0}_j(Z, W).$$

Note that if $Z$ admits a continuous self map $\phi: (Z, W) \to (Z, W)$ whose action $\phi_\ast: H_j(Z, W) \to H_j(Z, W)$ on homology  has an eigen-element $h$ such that $\phi_\ast(h) = \lambda \cdot h$ for a scalar $\lambda$, subject to $|\lambda| >  1$, then $$\| h\|_\D  \geq \| \phi_\ast(h)\|_\D = |\lambda|\cdot \| h\|_\D.$$ Thus  $\| h\|_\D = 0$; so such an element $h \in H_j(Z, W)$ dies in the quotient $H^{\mathbf \D}_j(Z, W)$.

In fact, the construction of $H^{\mathbf \D}_1(\sim)$ always produces a trivial result: for any $Z$, $H_1^{\mathbf \D}(Z) = 0$. However, if $Z$ is a closed surface of genus $g > 1$, then $H_2^{\mathbf \D}(Z) \neq 0$. Moreover, when $Z$ is a product of many closed surfaces of genera $g(M_i) \geq 2$, then $H_2^{\mathbf \D}(Z) \neq 0$ is rich.
\smallskip

If $f: Z \to Y$ is a continuous map of topological spaces, then $\| h \|_\D \geq \| f_\ast(h) \|_\D$. Therefore, $f$ induces a continuous  linear map of normed spaces $f_\ast: H^{\mathbf \D}_j(Z)  \to H^{\mathbf \D}_j(Y)$ whose operator norm is $\leq 1$.  It is not difficult to verify that when $f$ is a homotopy equivalence, then this map $f_\ast$ is an isometry (in the simplicial norm) between the normed spaces $H^{\mathbf \D}_j(Z)$ and $H^{\mathbf \D}_j(Y)$. In particular, the shape of the unit ball in $H^{\mathbf \D}_j(Z)$ is an invariant of the homotopy type of $Z$. 

We can take this observation one step further. The Mapping Theorem from \cite{Gr}, section 3.1, implies that the classifying map $f: Z \to K(\pi_1(Z), 1)$ induces an \emph{isometry} $f_\ast: H^{\mathbf \D}_j(Z) \to H^{\mathbf \D}_j \big(K(\pi_1(Z), 1)\big)$ (so that $f_\ast$ is a monomorphism). 
%{\bf WORK}
%We can take this observation one step further. The Mapping Theorem from \cite{Gr}, section 3.1, implies that the \emph{normed} spaces $\{H^{\mathbf \D}_j(Z)\}_j$ depend only on the fundamental group $\pi_1(Z)$, specifically on its bounded homology. Therefore $H^{\mathbf \D}_j(Z)$ is the pull-back of the group $H^{\mathbf \D}_j \big(K(\pi_1(Z), 1)\big)$ under the classifying map $f: Z \to K(\pi_1(Z), 1)$. %{\bf WORK}
\smallskip

\noindent{\bf Question 3.1.} For a given group $\pi$, how to describe the shape of the unit ball $B_j$ in the normed space $H^{\mathbf \D}_j \big(K(\pi, 1); \R\big)$ in terms of $\pi$ (say, in terms of its subgroups or in terms of the representation theory)? Is $B_j$ a polyhedron for a finitely-presented $\pi$? 
\hfill $\diamondsuit$
\smallskip

It turns out that formula (\ref{eq11.19}) gives a nontrivial lower boundary for the number of strata $\s \in \mathcal S^\bullet_F(DX)$ of dimension $n- k$ for \emph{any} shadow $F \in  \mathsf {Shad}(X,  \mathsf E \Rightarrow \mathsf T)$. This  universal boundary is constructed in terms of the group $H^{\mathbf \D}_{k+1}(DX)$. % and the norm $\|\sim\|_\D$ on it. 

The next theorem is also an expression of  the amenable localization, coupled with the Poincar\'{e} duality. It can be viewed as a version of \emph{Morse Inequalities}, where a Morse function $f: X \to \R$ is replaced by a shadow $F$, the $f$-critical points are replaced by the $F$-induced strata in the double $DX$, and the homology of $X$ by the ``homology" $H^{\mathbf \D}_\ast(DX)$. %, enhanced with a simplicial norm. 
It shows, in particular,  that the groups $H^{\mathbf \D}_{k+1}(DX)$ provide lower bounds of the ranks of the groups $\mathsf C^{n-k}_\mho(DX, F)$ from the differential complex in (3.9).% (\ref{eq11.10}).

\begin{theorem}\label{th11.4} Let $X$ be a compact connected $(n+1)$-dimensional $\mathsf{PL}$-manifold with a nonempty boundary.
\begin{itemize}  
\item Let, for each connected component of the boundary $\d X$, the image of its fundamental group in $\pi_1(DX)$ be amenable.
Then, for each $k \in [-1, n]$ and every shadow $F \in  \mathsf {Shad}(X,  \mathsf E \Rightarrow \mathsf T)$  the  Poincar\'{e} duality localizing operator $\mathcal L_{k+1}^{F, \mho}$ from (\ref{eq11.11}) has the property 
\begin{eqnarray}\label{eq11.21}
\textup{rk}\big(\textup{im} (\mathcal L_{k+1}^{F, \mho})\big) \geq \textup{rk}\big(H^{\mathbf \D}_{k+1}(DX)\big).
\end{eqnarray}
As a result,  $\textup{rk}\big(\mathsf C^{n-k}_\mho(DX, F)\big)$---the number of strata $\s \in \mathcal S^\bullet_F(DX)$ of dimension $(n- k)$---is greater than or equal to $\textup{rk}\big(H^{\mathbf \D}_{k+1}(DX)\big)$. \smallskip

\item Let, for each connected component of the boundary $\d X$, the image of its fundamental group in $\pi_1(X)$ be  amenable.
Then, for for each $k \in [0, n]$ and every shadow $F \in  \mathsf {Shad}(X,  \mathsf E \Rightarrow \mathsf T)$,  the Poincar\'{e} duality localizing operator $\mathcal M_{k+1}^{F, \mho}$ from (\ref{eq11.11a}) has the property 
\begin{eqnarray}\label{eq11.21a}
\textup{rk}\big(\textup{im} (\mathcal M_{k+1}^{F, \mho})\big) \geq \textup{rk}\big(H^{\mathbf \D}_{k+1}(X)\big).
\end{eqnarray}
Thus,  $\textup{rk}\big(\mathsf C^{n-k}_\mho(X, F^\circ)\big)$---the number of strata $\s \in \mathcal S^\circ_F(X)$ of dimension $(n- k)$---is greater than or equal to $\textup{rk}\big(H^{\mathbf \D}_{k+1}(X)\big)$.
\end{itemize}
\end{theorem}

\begin{proof} Put $\mathcal K^{\D = 0}_{k+1} =_{\mathsf{def}}  \mathcal K^{\D = 0}_{k+1}(DX)$. If $h \in H_{k+1}(DX)$ belongs to the kernel of the homomorphism $\mathcal L_{k+1}^{F, \mho}$, then evidently $|[\mathcal L_{k+1}^{F, \mho}(h)]|_\mho = |[0]|_\mho = 0$. By the inequality (\ref{eq11.18}), $\|h\|_\D = 0$. Therefore, $h \in \mathcal K^{\D = 0}_{k+1} $. In other words, $\ker (\mathcal L_{k+1}^{F, \mho}) \subset \mathcal K^{\D = 0}_{k+1}$. Hence,
$$\textup{rk}\big(H_{k+1}(DX)/ \mathcal K^{\D = 0}_{k+1}\big) \leq \textup{rk}\big(H_{k+1}(DX)/\ker (\mathcal L_{k+1}^{F, \mho})\big) = \textup{rk}\big(\textup{im} (\mathcal L_{k+1}^{F, \mho})\big).$$ 
So we get 
\begin{eqnarray}\label{eq11.22}
\textup{rk}\big(H^{\mathbf \D}_{k+1}(DX)\big) \leq \textup{rk}\big( \mathsf C^{n-k}_\mho(DX, F)/ \mathsf B^{n-k}_\mho(DX, F)\big)%\big(\mathsf Z^{n-k}_\mho(DX, F)\big) 
 \nonumber \\ \leq \textup{rk}\big(\mathsf C^{n-k}_\mho(DX, F)\big),
\end{eqnarray}
the later rank is the number of strata $\s$ of dimension $n-k$ in the stratification $ \mathcal S^\bullet_F(DX)$.
\smallskip

%Finally by \cite{Gr}, $\textup{rk}\big(H^{\mathbf \D}_{k+1}(DX)\big) = \textup{rk}\Big(H^{\mathbf \D}_{k+1}\big(K(\Pi, \, 1)\big)\Big)$.
Similar arguments, based on Theorem \ref{th11.3a}, prove that
\begin{eqnarray}\label{eq11.22a}
\textup{rk}\big(H^{\mathbf \D}_{k+1}(X)\big) \leq \textup{rk}\big( \mathsf C^{n-k}_\mho(X, F^\circ)/ \mathsf B^{n-k}_\mho(X, F^\circ)\big)%\big(\mathsf Z^{n-k}_\mho(DX, F)\big) 
 \nonumber \\ \leq \textup{rk}\big(\mathsf C^{n-k}_\mho(X, F^\circ)\big).
\end{eqnarray}
%Note that the number of strata of dimension $n-k$ in the stratification $\mathcal S^\circ_F(X)$ of is equal to number of strata  of dimension $n-k -1$ in the stratification $ \mathcal S_F(F(X))$. WRONG
\end{proof}

%\begin{remark}\label{rem11.1} 
\noindent {\bf Remark 3.1.} Since $H^{\mathbf \D}_{1}(DX) = 0$ and $H^{\mathbf \D}_{1}(X) = 0$ for all $X$, the statements of Theorem \ref{th11.4} are vacuous for  $k = 0$. Also, for $k = n$, the statement in the second bullet is vacuous since $H^{\mathbf \D}_{n+1}(X) = 0$.

\smallskip

Since the classifying map $\b: DX \to K(\Pi, 1)$, where $\Pi = \pi_1(DX)$, induces an isomorphism of the fundamental groups, $\b_\ast: H_\ast(DX) \to H_\ast(K(\Pi, 1))$ is an isometry (see \cite{Gr}). As a result, the induced map $\tilde\b_\ast: H^{\mathbf\D}_\ast(DX) \to H^{\mathbf\D}_\ast(K(\Pi, 1))$ is a monomorphism. Therefore,
\emph{the best lower estimate} that formula (\ref{eq11.22}) could provide is equal to $$\textup{rk}\big(H^{\mathbf \D}_{k+1}\big(K(\Pi, \, 1)\big)\big) \geq \textup{rk}\big(H^{\mathbf \D}_{k+1}(DX)\big).$$ 
Similarly, with $\pi = \pi_1(X)$, \emph{the best lower estimate} that formula (\ref{eq11.22a}) could provide is equal to $$\textup{rk}\big(H^{\mathbf \D}_{k+1}\big(K(\pi, \, 1)\big)\big) \geq \textup{rk}\big(H^{\mathbf \D}_{k+1}(X)\big).$$ 
\hfill $\diamondsuit$
%\end{remark}

\noindent{\bf Question 3.2.}
%\begin{question}\label{q11.1} 
Can one probe faithfully the shape of the unit ball $\tilde B_{k+1}$ in the quotient norm $\| \sim \|_\D$ on $H^{\mathbf \D}_{k+1}(DX)$ by counting the cardinalities $\#\Big( f(M) \cap DX^F_{-(k+1)} \Big)$ for various shadows $F$ and singular pseudo-manifolds $f: M \to DX$? Similar question can be posed about the shape of the unit ball $B_{k+1}$ in the quotient norm $\| \sim \|_\D$ on $H^{\mathbf \D}_{k+1}(X)$. 
\hfill $\diamondsuit$
%\end{question}
\smallskip

In one interesting special case that deals with the strata of maximal codimension Theorems \ref{th11.3}-\ref{th11.4}  
%and Corollary 3.1 % \ref{cor11.2}
can be made more explicit. The proposition below is a generalization of  Theorem 2 from \cite{AK}.

\begin{theorem}\label{th11.5} Let $X$ be a compact connected  and oriented $(n+1)$-dimensional $\mathsf{PL}$-manifold with a nonempty boundary. 

Assume that, for each connected component of the boundary $\d X$, the image of its fundamental group in $\pi_1(DX)$ is  amenable. 
Then there is a $X$-independent universal constant $\theta >  0$ such that, for any shadow $F \in  \mathsf {Shad}(X,  \mathsf E \Rightarrow \mathsf T)$,  the cardinality of the finite set $F(X)_{-n}$ satisfies the inequality 
\begin{eqnarray}\label{eq11.23}
 \#\big(F(X)_{-n}\big) \geq \theta \cdot \| [DX] \|_\D.
\end{eqnarray}
\end{theorem}

\begin{proof} The argument is based on the proof of Theorem 3.1.  %\ref{th11.3}. 
Note that the fundamental class $[DX] \in H_{n+1}(DX)$, being restricted to the regular neighborhood  
$$U =_{\mathsf{def}} \coprod_{\{\s \in \mathcal S^\bullet_F(DX)| \; \mu'(\omega_\s) = n \}} D^{n+1}_\s$$ of the locus $DX^F_{-(n+1)}$, equals the sum $$\sum _{\{\s \in \mathcal S^\bullet_F(DX)| \; \mu'(\omega_\s) = n \}} [D^{n+1}_\s, \d D^{n+1}_\s].$$ Therefore, by (\ref{eq11.15}) ,  $ [|\, [DX]_U\, |] = \#\Big(DX^F_{-(n+1)}\Big)$. Employing inequality (\ref{eq11.19}), we get 
$$\| [DX] |_\D \leq \Theta \cdot \#\Big(DX^F_{-(n+1)}\Big) \leq  \Theta \cdot \kappa \cdot \#\big(F(X)_{-n}\big),$$
where 
$\kappa = _{\mathsf{def}} \max_{\{\omega|\, \mu'(\omega) = n\}} (\mu(\omega) - \mu'(\omega))$
is the maximal cardinality of the fibers of the map $$F:\; DX^F_{-(n+1)} = F^{-1}\big(F(X)_{-n}\big) \cap \d X   \longrightarrow  F(X)_{-n}.$$
Choosing $\theta = (\Theta \cdot \kappa)^{-1}$ completes the argument. \smallskip

Note that $X$ has no absolute fundamental cycle and that $X^{F\circ}_{-(n+1)} = \emptyset$; therefore applying the second bullet of Theorem \ref{th11.4} leads to a  tautology.  
\end{proof}

The next implication of Theorem  \ref{th11.4} reveals the groups $H^{\mathbf \D}_{k+1}(DX)$ and  $H^{\mathbf \D}_{k}(X)$ as \emph{obstructions} to the existence of globally $k$-convex shadows of $X$.

\begin{corollary}\label{cor11.4} Let $X$ be a compact connected $(n+1)$-dimensional $\mathsf{PL}$-manifold with a nonempty boundary.  
Assume that, for each connected component of the boundary $\d X$, the image of its fundamental group in $\pi_1(X)$ is an amenable group.

If $X$ is globally $k$-convex, then the simplicial semi-norm is trivial on $H_{j+1}(DX)$  and on $H_{j}(X)$ for all $j \geq k$. 

In particular, if an oriented $X$ admits a shadow $F \in \mathsf {Shad}(X,  \mathsf E \Rightarrow \mathsf T)$ such that $F(X)_{-n} = \emptyset$, then $\| [DX] \|_\D = 0$ and $H^{\mathbf \D}_n(X) = 0$. %, and thus $\| [X/\d X]\|_\D = 0$.
\end{corollary}

\begin{proof} By Remark 3.2,  $F(X)_{-k} = \emptyset$ implies $F(X)_{-j} = \emptyset$ for all $j\geq k$.  Therefore $DX^F_{-(j+1)} = \emptyset$ for all $j \geq k$, since the $(n-j)$-dimensional strata from $\mathcal S_F^\bullet(DX)$ are contributed only by the strata of $F(X)$ of dimensions $(n-j)$ and $(n-j-1)$. The exception to this rule is provided by the $0$-dimensional strata from $\mathcal S_F^\bullet(DX)$: they are contributed by the points from $F(X)_{-n}$ only. 

In contrast, only the $(n-j)$-dimensional strata of $F(X)$ contributes to the $(n + 1- j)$-dimensional strata from $\mathcal S_F^\circ(X)$. So $F(X)_{-j} = \emptyset$ implies $X_{-j}^{F\circ} = \emptyset$. In turn, this implies $\mathsf C^{n+1 -j}_\mho(X, F^\circ) = 0$. Therefore, by (\ref{eq11.22a}), $H^{\mathbf \D}_{j}(X) = 0$ for all $j \geq k$.
%For any $h \in H_{k+1}(DX)$, we get $\|h\|_\D \geq \|g_\ast(h)\|_\D$, where $g: DX \to X/\d X$ is the map that collapses the second copy of $X$ to a point. Thus,  if the simplicial norm is nontrivial on $H_{k+1}(X/\d X)$, it is nontrivial on $H_{k+1}(DX)$. Hence $H^{\mathbf \D}_{k+1}(X/\d X) \neq 0$ implies that $H^{\mathbf \D}_{k+1}(DX) \neq 0$. {\bf WORK ????}
\end{proof}

\noindent{\bf Definition 3.8.}
%\begin{definition}\label{def11.14}
 Let $X$ be a compact connected $(n+1)$-dimensional $\mathsf{PL}$-manifold with a nonempty boundary, and $\mathcal A$ denotes an abelian group or a field. 
For any shadow $F \in  \mathsf {Shad}(X,  \mathsf E \Rightarrow \mathsf T)$ and each $j \in [0, n+1]$,
let  $$\Sigma c^j(X, F) =_{\mathsf{def}} \textup{rk}_\mathcal A \big(\mathsf C_\mho^j(DX, F)\big).$$ We call this integer the $j$-th \emph{suspension $\mho$-complexity} of the shadow  $F$.

Let $$\Sigma c^j_{\mathsf{shad}}(X) =_{\mathsf{def}} \min_{F \in  \mathsf {Shad}(X,  \mathsf E \Rightarrow \mathsf T)} \Sigma c^j(X, F).$$
We call  $\Sigma c^j_{\mathsf{shad}}(X)$ the $j$-th \emph{suspension shadow complexity} of $X$. \hfill $\diamondsuit$ 
%\end{definition}
\smallskip

These numbers can be organized into a sequence 
$$\mathbf{\Sigma c}_{\mathsf{shad}}(X) =_{\mathsf{def}} \big(\Sigma c^0_{\mathsf{shad}}(X),\, \Sigma c^1_{\mathsf{shad}}(X)\, \dots \, \Sigma c^{n+1}_{\mathsf{shad}}(X)\big).$$
Note that this optimal sequence may not be realizable by a single shadow! To avoid this fundamental difficulty, we will need a more manageable version of the previous definition.
\smallskip

\noindent{\bf Defintion 3.8.}
%\begin{definition}\label{def11.15} 
Let $X$ be a compact connected $(n+1)$-dimensional $\mathsf{PL}$-manifold with a nonempty boundary. For any shadow $F \in  \mathsf {Shad}(X, \mathsf E \Rightarrow \mathsf T)$, form the sequence 
$$
\mathbf{\Sigma c}(X, F) =_{\mathsf{def}} \big(\Sigma c^0(X, F),\, \Sigma c^1(X, F),\, \dots ,\, \Sigma c^{n+1}(X, F)\big).
$$
and take the \emph{lexicographic} minimum
$$\mathbf{\Sigma c}_{\mathsf{shad}}^{\mathsf{lex}}(X) =_{\mathsf{def}} \textup{lex.min}_{F \in  \mathsf {Shad}(X,  \mathsf E \Rightarrow \mathsf T)}\; \mathbf{\Sigma c}(X, F).$$

We denote by $\Sigma c_{\mathsf{shad}}^{\mathsf{lex},\, j}(X)$ the $(j+1)$ component of the vector $\mathbf{\Sigma c}_{\mathsf{shad}}^{\mathsf{lex}}(X)$ and call it the $j$-th \emph{lexicographic suspension shadow complexity} of $X$. \hfill $\diamondsuit$
%\end{definition}
\smallskip

\noindent{\bf Remark 3.2.}
%\begin{remark}\label{rem11.4} 
By its very definition, the lexicographically optimal sequence of complexities is delivered by some shadow $F$!

Evidently, for each $j$, $\Sigma c_{\mathsf{shad}}^{\mathsf{lex},\, j}(X)\,  \geq \, \Sigma c^j_{\mathsf{shad}}(X).$
\hfill $\diamondsuit$
%\end{remark}
\smallskip

Theorem 3.4 %\ref{th11.4} 
has an immediate implication. 

\begin{corollary}\label{cor11.5} Let $X$ be a compact connected $(n+1)$-dimensional $\mathsf{PL}$-manifold with a nonempty boundary.  
Assume that, for each connected component of the boundary $\d X$, the image of its fundamental group in $\pi_1(DX)$ is an amenable group. 

Then, for any $k \in [-1, n]$, the suspension shadow complexity of  $X$ satisfies the inequality
$$\Sigma c^{n-k}_{\mathsf{shad}}(X) \geq  \textup{rk}\big(H^{\mathbf \D}_{k+1}(DX)\big).$$ % = \textup{rk}\Big(H^{\mathbf \D}_{k+1}\big(K(\Pi, \, 1)\big)\Big).$$
%Dropping the assumption about  amenability of the images of the fundamental groups $\pi_1(\d X, pt)$ in $\pi_1(DX)$, we get a weaker inequality $$\Sigma c^{n-k}_{\mathsf{shad}}(X) \geq   \textup{rk}\big(H^{\mathbf \D}_{k+1}\big(X/ \d X\big)\big).$$
%%where  $\pi =_{\mathsf{def}}\pi_1(X/\d X)$.
\hfill $\diamondsuit$
\end{corollary}

%Put $\Pi =_{\mathsf{def}}\pi_1(DX)$ and   $H^{\mathbf \D}_j(\Pi)  =_{\mathsf{def}} H^{\mathbf \D}_j\big(K(\Pi, \, 1)\big)$. 
Let us consider the sequence 
\begin{eqnarray}
\qquad \mathbf{rk}\big(H^{\mathbf \D}_\ast(DX)\big)  =_{\mathsf{def}} \;  \Big( \textup{rk}\big(H^{\mathbf \D}_{n+1}(DX)\big), \,  \textup{rk}\big(H^{\mathbf \D}_{n}(DX)\big), \, \dots , \,  \textup{rk}\big(H^{\mathbf \D}_{0}(DX)\big)\Big) \nonumber.
\end{eqnarray}

%\begin{eqnarray}
%\mathbf{rk}\big(H^{\mathbf \D}_\ast(\Pi)\big)  =_{\mathsf{def}} \qquad \qquad \nonumber \\  \Big( \textup{rk}\big(H^{\mathbf \D}_{n+1}(\Pi)\big), \,  \textup{rk}\big(H^{\mathbf \D}_{n}(\Pi)\big), \, \dots , \,  \textup{rk}\big(H^{\mathbf \D}_{0}(\Pi)\big)\Big) \nonumber.
%\end{eqnarray}

Then Corollary \ref{cor11.5} can be expressed in its compressed form as
\begin{eqnarray}\label{eq11.24}
\mathbf{\Sigma c}^{\mathsf{lex}}_{\mathsf{shad}}(X) \, \geq \, \mathbf{\Sigma c}_{\mathsf{shad}}(X) \, \geq \, \mathbf{rk}\big(H^{\mathbf \D}_\ast(DX)\big),
\end{eqnarray}
where the vectorial inequality is understood as the inequality among all the the corresponding components of the participating vectors.\smallskip

Let $X$ be a compact connected $(n+1)$-dimensional $\mathsf{PL}$-manifold with a nonempty boundary, and $F: X \to K$ its shadow. Recall that, by Definition 3.1,  %\ref{def11.6},  
the cardinality of the fiber $F: X(F, \omega) \cap \d X \to K(\omega)$ depends only on $\omega \in \mathcal S$: it is the difference $\mu(\omega) - \mu'(\omega)$. 
\smallskip

\noindent{\bf Definition 3.9.}
%\begin{definition}\label{def11.16}  
Consider a filtered $CW$-complex $K \in \mathsf {Shad}(\mathsf T, n)$ and its filtration as in (\ref{eq11.2}). Put $$\kappa_{j}(n) =_{\mathsf{def}} \max_{\{\omega|\, \mu'(\omega) = n-j\}} \big(\mu(\omega) - \mu'(\omega)\big).$$

The \emph{weighted  $j$-th complexity} of $K$ is defined by the formula $$^\sharp c^j(K) =_{\mathsf{def}}  \kappa_j(n) \cdot c^j(K).$$
\hfill $\diamondsuit$
%\end{definition}

The next proposition helps to link the suspension complexities $\Sigma c^j(X, F)$ to the weighted ones $^\sharp c^j(F(X))$ and thus motivates Definition 3.9. %\ref{def11.16}.
%%%%%

\begin{corollary}\label{cor11.6} Let $X$ be a compact connected $(n+1)$-dimensional $\mathsf{PL}$-manifold with a nonempty boundary.  Assume that, for each connected component of the boundary $\d X$, the image of its fundamental group in $\pi_1(X)$ is an amenable group. 

Then, for any shadow $F \in  \mathsf {Shad}(X, \mathsf E \Rightarrow \mathsf T)$ and each $k \in [-1, n]$,
\begin{eqnarray}\label{eq11.25}
^\sharp c^{n-k}(F(X))  + 2\cdot\,  ^\sharp c^{n-k-1}(F(X))  \geq  \textup{rk}\big(H^{\mathbf \D}_{k+1}(DX)\big), \\
^\sharp c^{n-k-1}(F(X))  \geq  \textup{rk}\big(H^{\mathbf \D}_{k+1}(X)\big).  \nonumber
\end{eqnarray}
\end{corollary}

\begin{proof} Let $K = F(X)$ be a shadow of $X$. Any connected component of $K(\omega)$ of dimension $n- \mu'(\omega)$ (by Definition 3.1) gives rise to $\mu(\omega)- \mu'(\omega)$ strata $\s \in \mathcal S^\bullet_F(DX)$ of  dimension $n- \mu'(\omega)$ and to $2(\mu(\omega)- \mu'(\omega) -1)$ strata $\s \in \mathcal S^\bullet_F(DX)$ of dimension $n + 1- \mu'(\omega)$. Therefore (see Fig. 3) the number of strata $\s \in \mathcal S^\bullet_F(DX)$ of dimension $j = n-k$ is given by the formula:
\begin{eqnarray}\label{eq11.30a}
\sum_{\omega|\, \mu'(\omega) = k} (\mu(\omega) -k)\cdot \#\Big(\pi_0\big(K(\omega)\big)\Big)   + \sum_{\omega|\, \mu'(\omega) = k+1}2(\mu(\omega) -k -1)\cdot\#\Big(\pi_0\big(K(\omega)\big)\Big).\nonumber \\
\hfill 
\end{eqnarray}
By the definition of $\{\kappa_j(n)\}_j$, the latter number is smaller than or equal to
$$\kappa_{n-k}(n)\cdot \sum_{\omega|\, \mu'(\omega) = k}\#\Big(\pi_0\big(K(\omega)\big)\Big)   + 2\kappa_{n-k -1}(n)\cdot \sum_{\omega|\, \mu'(\omega) = k+1}\#\Big(\pi_0\big(K(\omega)\big)\Big),$$
where the fist sum is the number of strata in $F(X)$ of dimension $n-k$,  and the second sum is the number of strata of dimension $n-k-1$. Thus the previous formula is an upper bound of the number of components in $\mathcal S^\bullet_F(DX)$ of dimension $n-k$. 

Similar upper estimate of the number of connected components in $X^{F\circ}_{-k}$  leads to the LHS of the second inequality in (\ref{eq11.25}). 

Now Theorem \ref{th11.4} implies formulas (\ref{eq11.25}).
\end{proof} 

Corollary \ref{cor11.6} has the following two immediate implications which reveal the non-triviality of the groups/spaces $H^{\mathbf \D}_{k+1}(DX) \subset H^{\mathbf \D}_{k+1}(\Pi)$ and $H^{\mathbf \D}_{k}(X) \subset H^{\mathbf \D}_{k}(\pi)$ as \emph{obstructions} to the existence of shadows $F$ of low/vanishing complexity. The first implication is just a repackaging of Corollary \ref{cor11.4}. 

\begin{corollary}\label{cor11.7} Under the hypotheses of Corollary \ref{cor11.6}, 
%if $X$ has a shadow $F$ such that $c^{n-k}_\mho(F(X)) = 0$,  $c^{n-k -1}_\mho(F(X)) = 0$, then the simplicial semi-norm on $H_{k+1}(DX)$ is trivial. 
 if $H^{\mathbf \D}_{k+1}(DX) \neq 0,$ then either $c^{n-k}(F(X)) \neq 0$ or $c^{n-k -1}(F(X)) \neq 0$ for \emph{any} shadow $F$. \smallskip
 
If $H^{\mathbf \D}_{k}(X) \neq 0,$ then $c^{n-k -1}(F(X)) \neq 0$ for \emph{any} shadow $F$.
\hfill $\diamondsuit$
\end{corollary}

%{\bf WORK}

\begin{corollary} Let $X$ be a compact connected $(n+1)$-dimensional $\mathsf{PL}$-manifold with a nonempty boundary.  Assume that, for each connected component of the boundary $\d X$, the image of its fundamental group in $\pi_1(X)$ is amenable. Then, for each $k \in [-1, n]$,
\begin{eqnarray}\label{eq11.26}
^\sharp c^{n-k}_{\mathsf {shad}}(X, \mathsf E \Rightarrow \mathsf T)  + 2\cdot\,  ^\sharp c^{n-k-1}_{\mathsf {shad}}(X, \mathsf E \Rightarrow \mathsf T)   \geq  \textup{rk}\big(H^{\mathbf \D}_{k+1}(DX)\big),\\
^\sharp c^{n-k-1}_{\mathsf {shad}}(X, \mathsf E \Rightarrow \mathsf T)   \geq  \textup{rk}\big(H^{\mathbf \D}_{k+1}(X)\big).\nonumber
\end{eqnarray}
\hfill $\diamondsuit$
\end{corollary}

%%%%%
\section{Complexity of Traversing Flows}
%%%%%

Now let us examine applications of these results about the complexities of shadows to the traversing flows on smooth manifolds with boundary. In fact, the entire general setting in Section 3 was designed with these applications in mind. As we will show next,  any proposition about shadows of $\mathsf{PL}$-manifolds with boundary and the simplicial semi-norms has an analogue for the smooth manifolds with boundary that carry a traversally generic vector field. In short, the traversally generic fields are a good source of shadows. \smallskip 

We will apply Theorem \ref{th11.3}, Theorem \ref{th11.4}, and their corollaries to traversing vector fields $v$. In this setting, the poset $\mathcal S =_{\mathsf{def}} \Omega^\bullet$, $\mu : = | \sim |$, $\mu' : = | \sim |'$ (see (\ref{eq2.2})), the role of shadows is played by the obvious maps $\Gamma: X \to \mathcal T(v)$ which belong to an appropriate set $\mathsf{Shad}(X,  \mathsf E \Rightarrow \mathsf T)$. 

\begin{theorem}\label{th11.6} Let $X$ be a smooth compact connected and oriented $(n+1)$-manifold with boundary. Any traversally generic vector field $v$ gives rise to a shadow $\Gamma: X \to \mathcal T(v)$ in the sense of Definition 3.1. The the model projections $\{p_\omega: \mathsf E_\omega \to \mathsf T_\omega\}_{\omega \in  \Omega^\bullet_{'\langle n]}}$ are described in \cite{K3}, Theorems 7.4  and 7.5., utilizing special coordinates as in (2.4). 
\end{theorem}

\begin{proof} Recall that, for a traversally generic $v$, the trajectory space $\mathcal T(v)$ is a Whiney stratified space (see \cite{K4}, Theorem 2.2),  
which implies that $\Gamma$ is a simplicial map with respect to the appropriate triangulations of the source and target spaces. Moreover, the triangulation of $X$ can be chosen to be smooth. 

By the very definition of a traversing flow, each fiber $\Gamma^{-1}(\g), \, \g \in \mathcal T(v)$ is either a closed segment, or a singleton. 

By Corollary 5.1 from \cite{K3}, the map $\Gamma: \d_1X \to \mathcal T(v)$, being restricted to the preimage of each proper stratum $\mathcal T(v, \omega)$, is a cover with the trivial monodromy and fibers of cardinality $$\#(\sup(\omega)) = |\omega| - |\omega|' = \mu(\omega) - \mu'(\omega).$$

Moreover, by Theorem 2.1 from \cite{K4}, each pure stratum $\mathcal T(v, \omega)$ is an open orientable smooth manifold. 

By Lemma 3.4 from \cite{K2}, Theorems 5.2 and 5.3 from \cite{K3}, each point $\g \in \mathcal T(v, \omega)$, has a regular neigborhood $V_\g \subset \mathcal T(v)$, so that $\Gamma: \Gamma^{-1}(V_\g) \to V_\g$ is $\mathsf{PL}$-homeomorphic   to the model projection $p_\omega: \mathsf E_\omega \to \mathsf T_\omega$. 

This completes the checklist of bullets from Definition 3.1.
\end{proof}

Let $X$ be a smooth compact connected and oriented $(n+1)$-manifold with boundary, and $v$ a traversally generic vector field.  
  
Let us consider a filtration 
$$\mathcal T(v) = \mathcal T(v)_{-0} \supset \mathcal T(v)_{-1} \supset \dots \supset \mathcal T(v)_{-n}$$
of the trajectory space $\mathcal T(v)$ by the closed subcomplexes
\begin{eqnarray}\label{eq11.28}
\mathcal T(v)_{-j} =_{\mathsf{def}} \bigcup_{\big\{\omega \in \Omega^\bullet_{'\langle n]}\big |\; |\omega|' \geq j\big\}} \, \mathcal T(v, \omega)
\end{eqnarray}
of dimensions $n-j$.  
\smallskip

Let $\mathcal A$ be an abelian coefficient system on $\mathcal T(v)$ (equivalently, on $X$). As a default, $\mathcal A = \R$. For each $j \in [0, n]$, consider the relative homology groups\footnote{in our notations, we suppress the dependance of these homology groups on the coefficients $\mathcal A$. }
\begin{eqnarray}\label{eq11.29} 
\{\mathsf C^\mho_j(\mathcal T(v)) : = H_j\big(\mathcal T(v)_{-n+j}, \mathcal T(v)_{-n+j-1}; \, \mathcal A\big)\}_j  
\end{eqnarray}
associated with the filtration. As in the context of shadows, $\mathsf C^\mho_j(\mathcal T(v))$ is the top homology of the quotient $\mathcal T(v)_{-n+j}/\mathcal T(v)_{-n+j-1}$.

These homology groups can be organized into a \emph{differential complex}
\begin{eqnarray}\label{eq11.30}
\mathbf C^\mho_\ast(\mathcal T(v)) =_{\mathsf{def}} \qquad \qquad  \\ 
\big\{0 \to \mathsf C^\mho_n(\mathcal T(v)) \stackrel{\partial_n}{\rightarrow} \mathsf C^\mho_{n-1}(\mathcal T(v)) \stackrel{\partial_{n-1}}{\rightarrow}  \dots \stackrel{\partial_1}{\rightarrow}  \mathsf C^\mho_0(\mathcal T(v)) \to 0\big\}, \quad \nonumber
\end{eqnarray}
where the differentials $\{\partial_j\}$ are the boundary homomorphisms from the long exact homology sequences of the triples $\{\mathcal T(v)_{-n+j} \supset  \mathcal T(v)_{-n+j-1} \supset \mathcal T(v)_{-n+j-2}\}_j.$

The modules $\mathsf C^\mho_j(\mathcal T(v))$ are free and finitely generated, the number of generators being the number of connected components of the set $$\mathcal T(v)^\circ_{-n+j} =_{\mathsf{def}} \mathcal T(v)_{-n+j}\setminus \mathcal T(v)_{-n+j-1}.$$ This observation is valid, since by Theorem 2.1 from \cite{K4}, % \ref{th8.1}
the components of $\mathcal T(v)^\circ_{-n+j}$ are orientable smooth manifolds.

The differential complex (\ref{eq11.30}) gives rise to the $\mho$-\emph{homology groups} $\mathsf H^\mho_j(v) =_{\mathsf{def}}  \mathsf H^\mho_j(\mathcal T(v))$ of the traversally generic  field $v$. \smallskip

As in the category of shadows, the traversing fields give rise to several notions of complexity.
\smallskip

\noindent{\bf Definition 4.1.}
%\begin{definition}\label{def11.17} 
With any traversally generic vector field $v$ on a smooth compact connected and oriented $(n+1)$-manifold $X$ with boundary, we associate the ordered collection of ranks $$\Big\{tc_j(v) =_{\mathsf{def}} \textup{rk}_{\mathcal A}\big(\mathsf C^\mho_j\big(\mathcal T(v)\big)\big)\Big\}_{0 \leq j \leq n},$$ 
where the groups $\mathsf C^\mho_j(\mathcal T(v)))$ have been introduced in (\ref{eq11.29}).

We call $tc_j(v)$ \emph{the $j$-th  traversing complexity} (``$tc$" for short) of the field $v$.\smallskip
\hfill $\diamondsuit$
%\end{definition}
\smallskip

\noindent{\bf Remark 4.1.} 
%\begin{remark}\label{rem11.5} 
Reviewing (\ref{eq11.29}), we notice that $tc_{n-k}(v) = 0$ implies that $\mathcal T(v)^\circ_{-k} = \emptyset$ since $\mathcal T(v)^\circ_{-k}$ is an orientable manifold. Examining the local models $\{p_\omega: \mathsf E_\omega \to \mathsf T_\omega\}_\omega$, we observe that if a particular combinatorial type $\omega'$ is missing in a model, then all the smaller combinatorial types $\omega'' \prec \omega'$ (of greater codimensions) are missing as well (see \cite{K3}). Therefore $\mathcal T(v)^\circ_{-k} = \emptyset$ implies that $\{\mathcal T(v)_{-j} = \emptyset\}_{j \geq k}$.  As a result, if the complexity $tc_{n-k}(v) = 0$, then the complexities $tc_{n-j}(v) = 0$ for all $j \geq k$.  \hfill $\diamondsuit$
%\end{remark}
\smallskip

\noindent{\bf Definition 4.2.}
%\begin{definition}\label{def11.18} 
Let $X$ be a compact connected and oriented smooth $(n+1)$-manifold  with boundary. 
For each traversally generic field $v \in \mathcal V^\ddagger(X)$, we form the sequence of traversing complexities: 
$$\mathbf{tc}(v) =_{\mathsf{def}}\{tc_0(v),\, tc_1(v),\, \dots ,\, tc_n(v)\}$$

Consider the \emph{lexicographical minimum} 
\begin{eqnarray}\label{eq11.12} 
\mathbf {tc}^{\mathsf{lex}}(X) =_{\mathsf{def}} \textup{lex.min}_{\{v \in \mathcal V^\ddagger(X)\}}\; \mathbf{tc}(v)
\end{eqnarray}
We call this vector \emph{the lexicographic traversing complexity} of $X$. 

We denote by $tc^{\mathsf{lex}}_j(X)$ the $(j+1)$-component of the vector  $\mathbf {tc}^{\mathsf{lex}}(X)$. 
\hfill $\diamondsuit$
%\end{definition}
\smallskip

\noindent{\bf Remark 4.2.}
%\begin{remark}\label{rem11.6} 
If a compact connected and oriented smooth manifold $X$ of dimension $n+1$  is globally $k$-convex in the sense of  Definition 2.2, %\ref{def5.11} 
then evidently $tc^{\mathsf{lex}}_j(X) = 0$ for all $j \leq n-k$ in accordance with Remark 4.1. %REF
\hfill $\diamondsuit$
%\end{remark} 
\smallskip 

\noindent{\bf Remark 4.3.}
%\begin{remark}\label{rem11.7} 
Thanks to Theorem \ref{th11.6}, traversally generic fields produce a particular kind of shadows from the set $\mathsf{Shad}(X, \mathsf E\Rightarrow  \mathsf T)$. Therefore, $$\mathbf{c}^{\mathsf{lex}}_{\mathsf{shad}}(X) \leq \mathbf{tc}^{\mathsf{lex}}(X).$$
\hfill $\diamondsuit$
%\end{remark}

As with the shadows, we will need ``suspensions" of these notions and constructions. 

We consider the stratification $\Omega^\bullet(X, v)$ of $X$ by the connected components of the strata 
\begin{eqnarray}\label{eq11.32}
\qquad \qquad  \big\{X^\circ(v, \omega) =_{\mathsf{def}} X(v, \omega) \setminus (\partial X \cap X(v, \omega)),\; X^\partial(v, \omega) =_{\mathsf{def}} \partial X \cap X(v, \omega)\big\}_{\omega \in \Omega^\bullet_{'\langle n]}}.
\end{eqnarray}
and the $\tau$-invariant stratification $\Omega^\bullet(DX, v)$ of the double $DX$, which is induced by $\Omega^\bullet(X, v)$.

Let 
$$X^v_{-(k+1)} =_{\mathsf{def}} \Big(\bigcup_{\big\{\omega \big |\, |\omega|' \geq k+1\big\}} X^\circ(v, \omega)\Big) \bigcup \Big( \bigcup_{\big\{\omega \big |\, |\omega|' \geq k\big\}} X^\partial(v, \omega)\Big),$$

The stratifications $\Omega^\bullet(X, v)$ and $\Omega^\bullet(DX, v)$ give rise to the filtrations: 
$$X =_{\mathsf{def}} X^v_{-0} \supset X^v_{-1} \supset \dots \supset X^v_{-(n+1)},$$
$$DX =_{\mathsf{def}} DX_{-0}^v \supset DX_{-1}^v \supset \dots \supset DX_{-(n+1)}^v$$
 by the union of strata of a fixed codimension, each pure stratum being an open manifold.
 
Analogously to (3.7), %(\ref{eq11.8}), 
we consider the homology  and cohomology groups with coefficients in $\mathcal A = \R$ or $\mathcal A = \Z$:
 \begin{eqnarray}\label{eq11.33} 
\{\mathsf C^\mho_j\big(DX, v\big) =_{\mathsf{def}} H_j\big(DX^v_{j -n - 1},\, DX^v_{j-n}; \, \mathcal A\big)\}_j, \\ 
\{\mathsf C_\mho^j\big(DX, v\big) =_{\mathsf{def}} H^j\big(DX^v_{j -n - 1},\, DX^v_{j-n}; \, \mathcal A\big)\}_j.\nonumber 
\end{eqnarray}
As in (3.8) and (3.9), %(\ref{eq11.9}), 
they can be organized into a differential complex: 
\begin{eqnarray}\label{eq11.34}
\mathbf C^\mho_\ast(DX, v) =_{\mathsf{def}}%\qquad \qquad \qquad \nonumber \\ 
\big\{0 \to \mathsf C^\mho_{n+1}(DX, v) \stackrel{\partial_{n+1}}{\rightarrow} \mathsf C^\mho_n(DX, v) \stackrel{\partial_{n}}{\rightarrow}  \dots \stackrel{\partial_1}{\rightarrow}  \mathsf C^\mho_0(DX, v) \to 0\big\}, \nonumber \\ 
 \quad
\end{eqnarray}
where the differentials $\{\partial_j\}$ are the boundary homomorphisms from the long exact homology sequences of triples $\{DX^v_{-j +1} \supset  DX^v_{-j} \supset DX^v_{-j -1}\}_j.$

Similarly, we introduce the dual differential complex 
\begin{eqnarray}\label{eq11.35}
\mathbf C_\mho^\ast(DX, v) =_{\mathsf{def}}%\qquad \qquad \qquad \nonumber \\ 
\big\{0 \leftarrow \mathsf C_\mho^{n+1}(DX, v) \stackrel{\partial^\ast_{n+1}}{\leftarrow} \mathsf C_\mho^n(DX, v) \stackrel{\partial^\ast_{n}}{\leftarrow}  \dots \stackrel{\partial^\ast_1}{\leftarrow} \mathsf C_\mho^0(DX, v) \leftarrow 0\big\} \nonumber \\ 
 \quad
\end{eqnarray}

These differential complexes produce the suspension $\mho$-homology $\mathbf{H}^\mho_\ast\big(DX, v\big)$ and $\mho$-cohomology $\mathbf{H}_\mho^\ast\big(DX, v\big)$ of  traversally generic $v$-flows on $X$.\smallskip 

\noindent{\bf Definition 4.3.}
%\begin{definition}\label{def11.19} 
Let $X$ be a compact connected oriented and smooth $(n+1)$-dimensional manifold with a  boundary. 
For any traversally generic field  $v \in \mathcal V^\ddagger(X)$ and each $j \in [0, n+1]$,
let  $$\Sigma tc^j(X, v) =_{\mathsf{def}} \textup{rk}_\mathcal A \big(\mathsf C_\mho^j(DX, v)\big).$$ We call this integer the $j$-th \emph{suspension $\mho$-complexity} of the field  $v$.
\hfill $\diamondsuit$
%\end{definition} 
\smallskip

\noindent{\bf Definition 4.4.}
%\begin{definition}\label{def11.20} 
Let $X$ be a compact connected oriented and smooth $(n+1)$-dimensional manifold with a  boundary. For any  traversally generic field  $v \in \mathcal V^\ddagger(X)$, we form the sequence 
$$
\mathbf{\Sigma tc}(X, v) =_{\mathsf{def}} \big(\Sigma tc^0(X, v),\, \Sigma tc^1(X, v),\, \dots ,\, \Sigma tc^{n+1}(X, v)\big).
$$
and take the \emph{lexicographic} minimum
$$\mathbf{\Sigma tc}_{\mathsf{trav}}^{\mathsf{lex}}(X) =_{\mathsf{def}} \textup{lex.min}_{v \in \mathcal V^\ddagger(X)}\; \mathbf{\Sigma tc}(X, v).$$

We denote by $\Sigma tc_{\mathsf{trav}}^{\mathsf{lex},\, j}(X)$ the $(j+1)$ component of the vector $\mathbf{\Sigma c}_{\mathsf{trav}}^{\mathsf{lex}}(X)$ and call it the $j$-th \emph{lexicographic suspended traversing complexity} of $X$. \hfill $\diamondsuit$
%\end{definition}
\smallskip

\noindent{\bf Remark 4.4.}
%\begin{remark}\label{rem11.8}
By its very definition, the lexicographically optimal sequence of complexities is delivered by some traversally generic vector field!

Evidently,  $$\mathbf{\Sigma tc}_{\mathsf{trav}}^{\mathsf{lex}}(X)\,  \geq \, \mathbf{\Sigma c}_{\mathsf{shad}}^{\mathsf{lex}}(X),$$
where ``$\mathsf{shad}$" refers to the shadows, based on the list of model maps  $\{\mathsf E_\omega \to \mathsf T_\omega\}_{\omega \in \Omega^\bullet_{'\langle n]}}$ exemplifying the local structure of traversally generic flows.

A priori, both vectors, $\mathbf{tc}_{\mathsf{trav}}^{\mathsf{lex}}(X)$ and $\mathbf{\Sigma tc}_{\mathsf{trav}}^{\mathsf{lex}}(X)$, are invariants of the \emph{smooth} topological type of $X$, while $\mathbf{c}_{\mathsf{shad}}^{\mathsf{lex}}(X)$ and $\mathbf{\Sigma c}_{\mathsf{shad}}^{\mathsf{lex}}(X)$ are invariants of the $\mathsf{PL}$-topological type.
\hfill $\diamondsuit$
%\end{remark}
\smallskip
 
Consider the space $\mathcal V^\ddagger(X)$ of traversally generic vector fields on $X$ and its subspace $\mathcal V^\ddagger_{\mathsf{fold}}(X)$, formed by fields for which the multiplicity $m(a) \leq 2$ for any $v$-trajectory $\g$ at each point $a \in \g \cap \d X$. Locally, for such fields $v$,  $\Gamma: \d X \to \mathcal T(v)$ is a folding map.
\smallskip

\noindent{\bf Definition 4.5.}
 %\begin{definition}\label{def11.21} 
Let $X$ be a compact connected oriented and smooth $(n+1)$-dimensional manifold with a  boundary.

For any traversally generic field $v \in\mathcal V^\ddagger_{\mathsf{fold}}(X)$ we define the \emph{lexicographic fold complexity} of $X$  by
\[ 
\mathbf{tc}_{\mathsf{fold\, trav}}^{\mathsf{lex}}(X) =_{\mathsf{def}} \textup{lex.min}_{v \in \mathcal V^\ddagger_{\mathsf{fold}}(X)}\; \mathbf{tc}(v).
\] 
\hfill $\diamondsuit$
%\end{definition}
\smallskip

Clearly, $\mathbf{tc}_{\mathsf{fold\, trav}}^{\mathsf{lex}}(X)  \geq \mathbf{tc}_{\mathsf{trav}}^{\mathsf{lex}}(X),$ provided $\mathcal V^\ddagger_{\mathsf{fold}}(X) \neq \emptyset$. When $X$ is a $3$-fold, by Theorem 9.5 in \cite{K}, 
$$tc_{\mathsf{fold \, trav}}^0(X) = tc_{\mathsf{trav}}^0(X).$$

Besides the inequality above, we know little about the relation between $\mathbf{tc}_{\mathsf{fold\, trav}}^{\mathsf{lex}}(X)$ and $\mathbf{tc}_{\mathsf{trav}}^{\mathsf{lex}}(X)$.
\smallskip

%\smallskip

Now let us examine how the marriage of Amenable Localization and  Poincar\'{e} Duality works in the environment of traversing fields.

For each $k \in [-1, n]$, by composing the Poincar\'{e} duality map $\mathcal D$ with the localization to the subsets $DX^v_{-(k+1)} \subset DX$ and $X^{v\circ}_{-(k+1)} \subset X$,  we get  two localization transfer maps
$$\mathcal L^v_{k+1}: H_{k+1}\big(DX\big) \stackrel{\approx\mathcal D}{\longrightarrow} H^{n-k}\big(DX\big) \stackrel{i^\ast}{\rightarrow}  H^{n-k}\big(DX^v_{-(k+1)}\big),$$
$$\mathcal M^v_{k+1}: H_{k+1}\big(X\big) \stackrel{\approx\mathcal D}{\longrightarrow} H^{n-k}\big(X, \d X\big) \stackrel{i^\ast}{\rightarrow}  H^{n-k}\big(X^v_{-(k+1)},\, X^v_{-(k+2)}\cup (\d X \cap X^v_{-(k+1)}) \big).$$
whose targets can be identified with the quotients $\mathsf C^{n-k}_\mho(DX, v)/ \mathsf B^{n-k}_\mho(DX, v),$  and \hfill\break $\mathsf C^{n-k}_\mho(X, v^\circ)/ \mathsf B^{n-k}_\mho(X, v^\circ),$ respectively. %(\ref{eq11.35})
We denote by 
\begin{eqnarray}\label{eq11.36}
\mathcal L_{k+1}^{v, \mho}: H_{k+1}\big(DX\big) \to \mathsf C^{n-k}_\mho(DX, v)/ \mathsf B^{n-k}_\mho(DX, v), \\
\mathcal M_{k+1}^{v, \mho}: H_{k+1}\big(X\big) \to \mathsf C^{n-k}_\mho(X, v^\circ)/ \mathsf B^{n-k}_\mho(X, v^\circ) \nonumber
\end{eqnarray}
the resulting operators. 

Now Theorem \ref{th11.3}, being combined with Theorem \ref{th11.6}, delivers  

\begin{theorem}[{\bf Amenable localization of the Poincar\'{e} duality for traversing flows}]\label{th11.7} \qquad 

Let $X$ be a compact connected oriented and smooth $(n+1)$-manifold $X$ with boundary.  
\begin{itemize}
\item Assume that for each connected component of the boundary $\d X$, the image of its fundamental group in $\pi_1(DX)$ is amenable.
Then, for each for each $k \in [-1, n]$, there exists an universal constant $\Theta > 0$ such that, for any traversally generic vector field $v$, the space $\mathsf C^{n-k}_\mho(DX, v)/ \mathsf B^{n-k}_\mho(DX, v)$ admits a norm $|[\; \; ]|_\mho$ so that 
\begin{eqnarray}\label{eq11.37}
\|h\|_\D \leq \Theta \cdot \big|[\mathcal L_{k+1}^{v, \mho}(h)]\big|_\mho
\end{eqnarray}  
for any $h \in H_{k+1}(DX)$. Here  the Poincar\'{e} duality localizing operator $\mathcal L_{k+1}^{v, \mho}$ is introduced in (4.9), % (\ref{eq11.11}), 
and the universal constant $\Theta > 0$ depends only on the list of model maps $\{p_\omega: \mathsf E_\omega \to \mathsf T_\omega\}_{\omega \in \Omega^\bullet_{'\langle n]}}$ in the way that is described in formula (\ref{eq11.17}).   \smallskip

\item Let, for each connected component of the boundary $\d X$, the image of its fundamental group in $\pi_1(X)$ be an amenable group.
Then, for each $k \in [0, n]$, there exists an universal constant $\Theta^\circ > 0$ such that, for any traversally generic vector field $v$, the space $\mathsf C^{n-k}_\mho(X, v^\circ)/ \mathsf B^{n-k}_\mho(X, v^\circ)$ admits a norm $|[\; \; ]|_\mho$ so that 
\begin{eqnarray}\label{eq11.38}
\|h\|_\D \leq \Theta^\circ \cdot \big|[\mathcal M_{k+1}^{v, \mho}(h)]\big|_\mho
\end{eqnarray}  
for any $h \in H_{k+1}(X)$. Here  the Poincar\'{e} duality localizing operator  operator $\mathcal M_{k+1}^{v, \mho}$ is introduced in  (4.9), % (\ref{eq11.11}), 
and the constant $\Theta^\circ > 0$ depends only on the list of model maps $\{p_\omega: \mathsf E_\omega \to \mathsf T_\omega\}_{\omega \in \Omega^\bullet_{'\langle n]}}$ in the way that is described in formula  (\ref{eq11.17a}). \hfill $\diamondsuit$
\end{itemize}
\end{theorem}

Theorem \ref{th11.6}, together with Corollary \ref{cor11.2} and Corollary \ref{cor11.3}, produces

\begin{corollary}\label{cor11.10} Let an integral homology class $h \in H_{k+1}(DX)$ be realized my a singular oriented pseudo-manifold $f: M \to DX$, $\dim(M) = k+1$.  

Under the hypotheses of Theorem 4.2, %\ref{th11.7},  
the number of intersections of the cycle $f(M)$ with the locus $DX^v_{-(k+1)}$ is greater than or equal to $\Theta^{-1} \cdot \|h\|_\D$, provided that $f$ is in general position with the subcomplex $DX^v_{-(k+1)} \subset DX$. 

%Dropping the amenability hypotheses of Theorem \ref{th11.3}, 
If a class $h \in H_{k+1}(X)$ is represented by a singular pseudo-manifold $f: M \to X$, then $(\Theta^\circ)^{-1} \cdot \|h\|_\D$ gives a lower bound of the number of transversal intersections of the absolute cycle $f(M)$ with the locus $X^{v \circ}_{-(k+1)} =_{\mathsf{def}} \mathsf{int}(X) \cap X^v_{-(k+1)}$. %{\bf WORK} 
\hfill $\diamondsuit$
\end{corollary}

The following example is a product of the author's conversations with Larry Guth. \smallskip

\noindent{\bf Example 4.1.}
%\begin{example}\label{ex11.4} 
Consider a fibration $p: E \to M$ whose base is a closed oriented hyperbolic manifold of dimension $k+1$ and whose fiber $F$ is a closed manifold. Assume that the $(n+1)$-dimensional manifold $E$ is oriented and that $p$ admits a section $s: M \to E$. In the complement to $s(M)$, pick a smooth codimension zero manifold $V$ such that  $\{\pi_1(\d V, pt)\}_{pt \in \d V}$ are amenable groups. Let $X = E \setminus \textup{int}(V)$. 

Recall that $\|[M]\|_\D$ is proportional to the hyperbolic volume $\textup{vol}(M)$ with an universal positive proportionality constant \cite{Gr}. Since $s$ is a section, and the simplicial semi-norm does not increase under the continuous maps, it follows that the simplicial norm of $s_\ast([M]) \in H_{k+1}(X)$ is proportional to $\textup{vol}(M)$ with the same proportionality constant.  

For any traversally generic vector field $v$ on $X$, consider the locus $X^v_{-(k+1)}$. We can perturb the section $s$ so that  the intersections of the cycle $s(M)$ is transversal to  $X^v_{-(k+1)}$. 

According to Corollary \ref{cor11.10}, there exits a universal constant $\rho > 0$ so that, for any such $X$ and any traversally generic vector field $v$ on $X$,  the transversal intersections of the perturbed cycle $s(M)$ with the locus $X^v_{-(k+1)}$ has $\rho \cdot \textup{vol}(M)$  intersections at least.

Since by the choice of $V$, $s(M) \cap \d V = \emptyset$, it follows that there exist at least $\rho \cdot \textup{vol}(M)$ trajectories $\g$ of the reduced multiplicity $m'(\g) = k+1$, each of which has a nonempty intersection  with the section $s(M) \subset \textup{int}(X)$ (and evidently with $\d X = \d V$). 
\hfill $\diamondsuit$
%\end{example}
\smallskip

%{\bf WORK; keep this !!}

%We can reinterpret Corollary 4.1 %\ref{cor11.6} 
%in terms of differential forms. Recall a standard form of the Poincar\'{e} duality on a compact smooth oriented $(n+1)$-manifold $Y$: for  any absolute $(k+1)$-cycle $\Sigma \subset Y$, the exists a closed differential $(n-k)$-form $\theta_\Sigma$ which vanishes on $\d Y$ and  such that
%$$\int_{\Sigma}  \tau = \int_Y \theta_\Sigma \wedge \tau.$$
%for any closed $(k+1)$-form $\tau$. 

%\begin{corollary}\label{cor11.11} Let a homology class $h \in H_{k+1}(DX; \Z)$ be realized by a singular closed oriented smooth  manifold $f: M \to DX$, $\dim(M) = k+1$.  

%Under the hypotheses of Theorem \ref{th11.7}, there exists an universal constant $K > 0$ such that  
%$$\int_{X^v_{-(k+1)}} \theta_{f(M)} \geq K \cdot \| [f(M)]\|_\D$$
%for all $(n+1)$-dimensional $X$, traversally generic vector fields $v$ on them, and $(k+1)$-cycles $f: M \to DX$. Here the closed $(n-k)$-form $\theta_{f(M)}$ on $DX$ is the Poincar\'{e} dual of the cycle $[f(M)]$. ????
%\hfill $\diamondsuit$
%\end{corollary}

%{\bf END WORK}

It worth describing the important special case ``$h = [DX]$" of formula (3.16), %\ref{eq11.16} 
being applied to the shadows that are produced by traversally generic flows. %\footnote{Compare the next theorem with the proposition that results by applying Theorem \ref{th11.7} to the fundamental class $h = [DX]$.}. 

\begin{theorem}\label{th11.8} Let $X$ be a smooth compact connected and oriented $(n+1)$-manifold with boundary.  
 Assume that, for each connected component of the boundary $\d X$ the image of its fundamental group in $\pi_1(DX)$ is an amenable group.

Let $v$ be a traversally generic vector field on  $X$. Then %{\bf ????}
\begin{eqnarray}\label{eq11.38}
\sum_{\omega \in \Omega^\bullet | \, |\omega|' = n} \; \theta(\omega) \cdot \#\big(\mathcal T(v, \omega)\big)\; \geq \; \big \|[DX]\big\|_\Delta,
\end{eqnarray}
where the universal constant $$\theta(\omega) =_{\mathsf{def}} \big\|[D\mathsf E_\omega,  D(\delta \mathsf E_\omega)] \big\|^{\Omega^\bullet_{'\langle n]}}_\D$$ is the $\Omega^\bullet_{'\langle n]}$-stratified simplicial volume of the model pair $(D\mathsf E_\omega,  D(\delta \mathsf E_\omega))$. 
%In particular, $\theta(\omega_\star) \geq 2^n + n$, where $\omega_\star =_{\mathsf{def}} (1\underbrace{2 \dots 2}_{n}1)$. {\bf ????}
\hfill $\diamondsuit$
\end{theorem}

The previous theorem represents a slight modification/improvement of Theorem 2 from \cite{AK}; below we state it for the reader's convenience.   

\begin{theorem}{\bf(Alpert, Katz)}\label{th11.9} Let $X$ be a smooth compact connected and oriented $(n+1)$-manifold with boundary.  
 Assume that, for each connected component of the boundary,  the image of its fundamental group in $\pi_1(DX)$ is an amenable group.
Let $v$ be a traversally generic vector field on  $X$. 

Then there is an universal constant $\rho(n) >  0$ such that, for any $X$ and $v$,  the cardinality of the set $\mathcal T(v)_{-n}$, formed by the trajectories of the maximal reduced multiplicity $n$,  satisfies the inequality 
\begin{eqnarray}\label{eq11.39}
 \#\big(\mathcal T(v)_{-n}\big)\, \geq \, \rho(n) \cdot \| [DX] \|_\D.
\end{eqnarray}
\hfill $\diamondsuit$
\end{theorem}

\noindent{\bf Example 4.2.}
%\begin{example}\label{ex11.6} 
Take the surface $X \subset \R^2$ and the constant vertical field $v = \d_y$ on it as in Fig. 1. Since $X$ is a disk with $4$ holes, the double $DX$ is a closed surface of genus $4$. It admits a hyperbolic metric. By \cite{Gr}, it follows that  $\| [DX] \|_\Delta = 2(2\cdot 4 - 2) = 12.$
%By Lemma \ref{lem11.2}, 
%$$2\| [X,  \d X] \|_\D \geq 12 \geq \| [X/ \d X] \|_\D.$$
%Thus $6 \leq \| [X,  \d X]\|_\D$ and $ \| [X/ \d X] \|_\D \leq 12$. 

At the same time, $\#(DX^v_{-2}) = 30$ and $\#(\mathcal T(v)_{-1}) = 12$ (see Fig. 1 and Fig. 3). Since $\rho(1) \leq \#\big(\mathcal T(v)_{-n}\big)/ \| [DX] \|_\D$ for any connected compact surface $X$ such that $\| [DX] \|_\D \neq 0$ and  a traversally generic  $v$ on it, the universal constant $\rho(1)$  in  (\ref{eq11.39})  gets  a  bound:  $\rho(1) \leq 12/12 = 1$.

However, the field $v$ in Fig. 1 is not the ``optimal" one for $X$.  To optimize $v$, take the radial field on an annulus $A$. Delete tree convex disks from $A$ to form $X$ and restrict the radial field  to $X$.  For the restricted field $v$ on $X$, the graph $\mathcal T(v)$ is trivalent with 6 verticies. So we get  $\#(\mathcal T(v)_{-2}) = 6$ and $\#(DX^v_{-2}) = 18$. As a result, the universal constant must satisfy the inequality $\rho(1)\leq 6/12 = 1/2$. Similar calculations apply to any $2$-disk with $q \geq 2$ holes to produce $\rho(1)\leq 1/2$. Perhaps, $\rho(1) = 1/2$. \hfill $\diamondsuit$
%If we analyze carefully the arguments in the proof of Theorem \ref{th11.3} surrounding formula (\ref{eq11.16}), we will see that they give a weaker upper bound of $const(1)$.
%\end{example}
\smallskip

Although desirable, an exact computation of the universal constants $\{\theta(\omega)\}$ in formula (\ref{eq11.38})
is somewhat tricky. We can estimate the number of cells-strata of the top dimension $|\omega|' $ in the model space $\mathsf T_\omega$.   It looks that this number can be calculated by counting the chambers in which the discriminant of each polynomial  $$\wp_i(u, \vec x) =_{\mathsf{def}} (u-i)^{\omega_i}  + \sum_{j=0}^{\omega_i -2} x_{i,j} (u -i)^j$$ divides the space $\R^{\omega_i -1}$. 

For $\omega_\star =_{\mathsf{def}} (1\underbrace{2 \dots 2}_{n}1)$, the model space $\mathsf T_{\omega_\star}$ is formed by attaching $n$ copies of a half-disk $D^n_+$ to a disk $D^n$ along the coordinate hyperplanes which divide $D^n$ into $2^n$ quadrants. So the number $\kappa(\omega_\star)$ of $n$-dimensional cells-strata in $\mathsf T_{\omega_\star}$ is $2^n + n$ (see \cite{K3}).  Therefore, the pull-back of the $\Omega^\bullet$-stratification in $\mathsf T_{\omega_\star}$ divides the model space $\mathsf E_{\omega_\star}$ into $2^n + n$ chambers-cells of dimension $n+1$. As a result, the stratified simplicial norms satisfy the inequalities
$$\big\| [\mathsf E_{\omega_\star}, \delta \mathsf E_{\omega_\star}] \big \|_\D^{\Omega^\bullet_{'\langle n]}} \geq 2^n + n, \qquad \big\| [D\mathsf E_{\omega_\star}, D(\delta \mathsf E_{\omega_\star})]\big \|_\D^{\Omega^\bullet_{'\langle n]}} \geq 2^{n+1} + 2n.$$

It seems that $\kappa(\omega) < \kappa(\omega^\star)$ for any minimal $\omega \neq \omega^\star$. \smallskip

Of course, it is more desirable to get  \emph{lower} bounds of these stratified simplicial norms...
\smallskip
 
Combining Theorem \ref{th11.4} with Theorem \ref{th11.6}, we get the following proposition. 

\begin{theorem}\label{th11.10} Let $X$ be a smooth compact connected and oriented manifold with boundary, $\dim(X) = n+1$.  Let $v$ be a traversally generic vector field on  $X$.
\begin{itemize}
\item Assume that, for each connected component of the boundary $\d X$, the image of its fundamental group in $\pi_1(DX)$ is an amenable group.
Then, for each $k \in [-1, n]$,
\begin{eqnarray}\label{eq11.40}
\textup{rk}\big(\textup{im} (\mathcal L_{k+1}^{v, \mho})\big)\, \geq \,  \textup{rk}\big(H^{\mathbf \D}_{k+1}(DX)\big).
\end{eqnarray}
As a result,  $\textup{rk}\big(\mathsf C^{n-k}_\mho(DX, v)\big)$---the number of $(n- k)$-dimensional strata in the stratification $\Omega^\bullet(DX, v)$\footnote{see $(11.32)$} of  $DX$%(\ref{11.32})) 
---is greater than or equal to $\textup{rk}\big(H^{\mathbf \D}_{k+1}(DX)\big)$.\smallskip

\item Assume that, for each connected component of the boundary $\d X$, the image of its fundamental group in $\pi_1(X)$ is an amenable group.
Then, for each $k \in [-1, n-1]$,  
\begin{eqnarray}\label{eq11.40}
\textup{rk}\big(\textup{im} (\mathcal M_{k+1}^{v, \mho})\big)\, \geq \,  \textup{rk}\big(H^{\mathbf \D}_{k+1}(X)\big).
\end{eqnarray}
Thus,  $\textup{rk}\big(\mathsf C^{n-k}_\mho(X, v^\circ)\big)$---the number of $(n- k)$-dimensional strata in the stratification $\Omega^\bullet(X, v^\circ)$---is greater than or equal to $\textup{rk}\big(H^{\mathbf \D}_{k+1}(X)\big)$.  \hfill $\diamondsuit$

\end{itemize}
\end{theorem} 

The first bullet  in Theorem \ref{th11.10}, together with the definitions of suspension complexities, leads to

\begin{theorem}\label{th11.11}  Let $X$ be a smooth compact connected and oriented manifold with boundary, $\dim(X) = n+1$. Assume that, for each connected component of the boundary $\d X$, the image of its fundamental group in $\pi_1(DX)$ is an amenable group. Then, for each $k \in [-1, n]$,
\begin{eqnarray}\label{eq11.41}
\mathbf{\Sigma tc}^{\mathsf{lex}}_{\mathsf{trav}}(X)\, \geq \, \mathbf{\Sigma c}^{\mathsf{lex}}_{\mathsf{shad}}(X) \, \geq \, \mathbf{rk}\big(H^{\mathbf \D}_\ast\big (DX\big)\big),
\end{eqnarray}
where the vectorial inequality is understood as the inequality among all the the corresponding components of the participating vectors. The vector $\mathbf{rk}\big(H^{\mathbf \D}_\ast\big (DX\big)\big)$ has been introduced prior to formula (3.23). %(\ref{eq11.24})
 \hfill $\diamondsuit$
\end{theorem}

Recall that, for $(n+1)$-dimensional $X$ and a traversally generic $v$ on it, $\# \sup(\omega) \leq n+2$  for all $\omega$'s (\cite{K2}). By the construction of the stratification $\Omega^\bullet(DX, v)$ of the double $DX$ (see Fig. 1 and Fig. 3), we get a formula which is similar to (\ref{eq11.30a}).  It connects the combinatorics of $\Omega^\bullet(DX, v)$ to the combinatorics of the stratification $\Omega^\bullet(\mathcal T(v))$ of $\mathcal T(v)$ by the connected components of the strata $\mathcal T(v, \omega)$:
\begin{eqnarray}
\# \pi_0\big(DX^v_{-j}\big) = \sum_{\{\omega | \, |\omega|' = j\}} \#\sup(\omega) \cdot \# \pi_0(\mathcal T(v, \omega)) \nonumber \\
 +\, 2 \cdot  \sum_{\{\hat\omega | \, |\hat\omega|' = j+1\}} (\#\sup(\hat\omega) -1) \cdot \# \pi_0(\mathcal T(v, \hat\omega))\nonumber \\
 \leq \; (n+2) \Big(\sum_{\{\omega | \, |\omega|' = j\}} \# \pi_0(\mathcal T(v, \omega)) + 2 \cdot \sum_{\{\hat\omega | \, |\hat\omega|' = j+1\}} \# \pi_0(\mathcal T(v, \hat\omega)) \Big)\nonumber \\
 = \; (n+2)\big( \# \pi_0\big(\mathcal T(v)^\circ_{-j}\big) + 2\cdot \# \pi_0\big(\mathcal T(v)^\circ_{-(j+1)}\big)\big). \nonumber 
\end{eqnarray}

Therefore, the suspended complexity of $v$ can be estimated in terms of the $v$-complexities: 
$$\Sigma tc^{n+1-j}(X, v) \leq (n+2)\big(tc^{n-j}(v)  + 2 \cdot tc^{n-j-1}(v)\big).$$
By  Theorem \ref{th11.10} and under its hypotheses, for each $k$, we get 
\begin{eqnarray}\label{eq11.42A}
\textup{rk}\big(H^{\mathbf \D}_{k+1}(DX)\big) \leq (n+2)\big(tc^{n-k}(v)  + 2 \cdot tc^{n-k-1}(v)\big), \\
\textup{rk}\big(H^{\mathbf \D}_{k}(X)\big) \leq (n+2)\big(tc^{n-k}(v)\big). \nonumber
\end{eqnarray} 

We notice that if $tc^{n-k}(v) = 0$, then by Remark 4.1, % \ref{rem11.5}, 
$tc^{n-k-1}(v) = 0$. Therefore,  $tc^{n-k}(v) = 0$ implies $H^{\mathbf \D}_{k+1}(DX) = 0$ and $H^{\mathbf \D}_{k}(X) = 0$.
\smallskip

\noindent{\bf Example 4.3.} 
%\begin{example} 
Let $\{M_i\}_{1\leq i\leq N}$ be several closed orientable surfaces of  genera $g(M_i) \geq 2$. Consider the connected sum $$Y =_{\mathsf{def}} (M_1\times S^1)\, \#\, (M_2\times S^1)\, \# \,\dots \#\, (M_N\times S^1),$$  and let $Z = Y \setminus B^3$, the complement to a smooth $3$-ball. 
Put $\Sigma(M_i)  =_{\mathsf{def}} M_i \times S^1$.  % $/ (pt_i \times S^1)$ denote the suspension of $M_i$.  Each product $M_i \times S^1$ is canonically mapped onto $\Sigma(M_i)$. 

We may assume that the $1$-handles $\{H_i \approx S^2 \times D^1\}_{1\leq i \leq N-1}$, participating in the connected sum construction of $Y$, are attached to the complements to the surfaces $M_i \times \star_i $ in $M_i \times S^1$ and to the complement of $B^3$. 

Consider a map $f_j : Z \to \Sigma(M_j)$ which %takes the $j$-th copy of $M_j \times S^1 \subset Z$ to $\Sigma(M_j)$  and 
is a homeomorphism on $M_j\times \star_i$. We construct $f_j$ in stages. First, we map each $1$-handle $H_i$ $ \subset Y$ to its core $D^1_i$ so that $Y$ is mapped to the union $W$ of $\Sigma(M_i)$'s to which the $1$-cores $D^1_i$'s are attached; each core is attached at a point $a_i \in \Sigma(M_i)$ and at a point $b_{i+1} \in \Sigma(M_{i+1})$. Under $f_j$, each of the two $3$-disks, $D^3_{a_i}$ and $D^3_{b_{i+1}}$, from $D^3 \times \d D^1_i$ is mapped to its center $a_i$ or $b_{i+1}$. Finally for each $j \in [2, N-1]$, we map the complement  $Y \setminus \big(\Sigma(M_j) \setminus (D^3_{b_j} \cup D^3_{a_j})\big)$ to $a_j \coprod b_{j}$.  For $j = 1$ and $j = N$, the constructions of $f_1$ and $f_N$ are similar.  

For each $j$, by an isotopy argument, we may assume that the ball $B^3$ belongs to $Y \setminus \Sigma(M_j)$ and then restrict $f_j$ to $Z$.    %{\bf WORK}

We claim that the $2$-dimensional classes $\{I_\ast[M_i] \in H_2^{\mathbf \D}(Z)\}_i$, where $I: \coprod_i M_i  \to Z$ is the obvious embedding,  are linearly independent. Indeed, assume that some combination $h =_{\mathsf{def}} \sum_i r_i I_\ast[M_i]$ has the property $\|h\|_\D = 0$. Then $\|(f_j)_\ast (h)\|_\D =0$. On the other hand, $$(f_j)_\ast (h) = r_j\cdot (I_j)_\ast[M_j] \in H_2(\Sigma(M_j)) \approx \R,$$ where the imbedding $I_j: M_j \subset \Sigma(M_j)$, being composed with $f_j$, produces the identity map of $M_j$. % induces an isomorphism of the fundamental groups. 
Therefore by the property of the simplicial semi-norm not to increase under continuous maps \cite{Gr}, $\|(I_j)_\ast[M_j]\|_\D =  \|[M_j]\|_\D \neq 0$. Thus $r_j = 0$. \smallskip

Let $X$ be a compact smooth $3$-fold which is homotopy equivalent to $Z$ and has a spherical boundary. Let $v$ be a traversally generic field on $X$. Since $\pi_1(\d X) = 0$, by Theorem \ref{th11.10} and following the arguments that lead to (\ref{eq11.30a}) and (\ref{eq11.42A}),  we count the $1$-dimensional connected strata in the stratification $\Omega^\bullet(DX, v)$ to get the inequality
\[
6 \cdot \#\mathcal T(v, 1221) + 2 \cdot \#\mathcal T(v, 13) + 2 \cdot \#\mathcal T(v, 31)\]
\[
 + \, 3 \cdot \#\pi_0(\mathcal T(v, 121)) + \#\pi_0(\mathcal T(v, 2))\, \geq \, \textup{rk}\big(H_2^{\mathbf \D}(DZ)\big) \geq 2N.
\]
Here the coefficients  6, 2, 2, 3, and 1 next to the cardinalities are determined by the cardinalities of the support of the corresponding combinatorial types $\omega = 1221, 13, 31, 121, 2$.

Note that $\|[DX]\|_\D = 0$. So, for $k + 1 = 3$, Theorem \ref{th11.10} provides a trivial estimate for 
 \[
6 \cdot \#\mathcal T(v, 1221) + 2 \cdot \#\mathcal T(v, 13) + 2 \cdot \#\mathcal T(v, 31),
\]
the number of $0$-dimensional strata in $\Omega^\bullet(DX, v)$.  \hfill $\diamondsuit$
%\end{example}
\smallskip

Let $\pi =_{\mathsf{def}} \pi_1(X)$ and $\Pi =_{\mathsf{def}} \pi_1(DX)$ . Utilizing formula (\ref{eq11.42A}), Theorems 4.2 %\ref{th11.10} 
and 4.5  %\ref{th11.11} 
expose the groups $H^{\mathbf \D}_{k+1}(DX) \subset H^{\mathbf \D}_{k+1}(\Pi) $ and  $H^{\mathbf \D}_{k}(X) \subset H^{\mathbf \D}_{k}(\pi)$ as \emph{obstructions} to the existence of globally $k$-convex traversing flows on $X$.

\begin{corollary}\label{cor11.11} Let $X$ be a smooth compact connected and oriented  $(n+1)$-manifold with boundary. Assume that $X$ admits a globally $k$-convex  traversally generic vector field $v$. 

If, for each connected component of the boundary $\d X$, the image of its fundamental group in $\pi_1(DX)$ is amenable, then the simplicial semi-norm on $H_{j+1}(DX)$ is trivial  for all $j \geq k$. 

If,  for each connected component of the boundary $\d X$, the image of its fundamental group in $\pi_1(X)$ is  amenable, then he simplicial semi-norm is trivial on $H_{j}(X)$ 
for all $j \geq k$. 

In particular, if $X$ admits a  vector field such that $\mathcal T(v)_{-n} = \emptyset$, then $\| [DX]\|_\D = 0$ and $H_n^{\mathbf \D}(X) = 0$. \hfill $\diamondsuit$
\end{corollary}

The next theorem, proven in \cite{AK},  should be compared with Theorem 7.5 from \cite{K}, %\ref{th3.6}, 
its older $3D$-relative. In a way, this theorem is the source of motivation for developing the machinery of amenable localization in \cite{AK} and in the present paper. For $3$-folds $M$ with a simply-connected boundary, Theorem 4.7 below is  known with $c(2) = 1/ \textup{Vol}(\Delta^3)$, the inverse of the volume of the ideal tetrahedron in the hyperbolic space (see \cite{K}).

\begin{theorem}[{\bf Alpert, Katz}]\label{th11.12} 
Let $M$ be a closed, oriented \emph{hyperbolic} manifold of dimension $n + 1 \geq 2$.  Let $X$ be the space obtained by deleting a domain $U$ %finitely many open balls 
from $M$, such that $U$ is contained in a ball $B^{n+1} \subset M$. %and the images of the fundamental groups $\{\pi_1(\d U, \star)\}_{\star \in \d U}$ in $\pi_1(X)$ are amenable\footnote{It suffices to assume that all the groups $\{\pi_1(\d U, \star)\}_{\star \in \d U}$ are amenable.}.  
Let $v$ be a traversally generic vector field on $X$.  

Then the cardinality of the set of $v$-trajectories of the maximal reduced multiplicity $n$ satisfies the inequality
\begin{eqnarray}\label{eq11.42}
\#\big(\mathcal T(v)_{-n}\big) \geq c(n) \cdot \textup{Vol}(M),
\end{eqnarray}
where $c(n) > 0$ is an universal constant, and $\textup{Vol}(M)$ denotes the hyperbolic volume of $M$.  

\hfill $\diamondsuit$
\end{theorem}

\begin{corollary}\label{cor11.13} The inequality (\ref{eq11.42}) of Theorem \ref{th11.12} is valid for any $X$, obtained by deleting a number of $(n+1)$-balls from a closed hyperbolic manifold $M$.
\end{corollary}

\begin{proof} When the domain $U$ from Theorem \ref{th11.12} is a union of $(n+1)$-balls, we can incapsulate them into a single $(n+1)$-ball $B$. %The boundary $\d U$ is simply-connected when $n \geq 2$ and its fundamental groups are  cyclic when $n = 1$. In any case, the fundamental groups $\{\pi_1(\d U, \star)\}_{\star \in \d U}$ are amenable.  
Therefore, Theorem \ref{th11.12} is applicable to $X = M \setminus U$.
\end{proof}

Let us apply Corollary \ref{cor11.13} to the landscape of Morse functions on closed hyperbolic manifolds. Moving towards this goal,  we formulate the following 
%\smallskip

%Informed by the considerations in Chapter 10, 

\begin{conjecture}\label{conj11.1} Let $f: M \to \R$ be a Morse function on a closed manifold $M$ of dimension $(n+1)$, and $v$ its gradient-like field. Assume that $v$ satisfies the Morse-Smale transversality condition (\cite{S}, \cite{S1}). Then, for all sufficiently small $\e > 0$, the  field $v$ on $$X =_{\mathsf{def}} M \setminus \coprod_{x \in \Sigma_f} B_\e(x)$$ is traversally generic.  The combinatorial types of the $v$-trajectories on $X$ are drawn from the list:
$(11), (121), (1221), \dots , (1\underbrace{2 \dots 2}_{n}1).$

Moreover, there exists a universal constant $c(n) > 0$ such that the number of broken $v$-trajectories on $M$, comprising $(n+1)$ segments, and the number of $n$-tangent $v$-trajectories in $X$ are related by the formula 
\begin{eqnarray}\label{eq11.44}
\#\big(\mathcal T(v, \omega_\star)\big) = c(n) \cdot \#\big(\mathsf{broken}_{(n+1)}(v)\big), 
\end{eqnarray} 
where $\omega_\star = (1\underbrace{2 \dots 2}_{n}1).$
\hfill $\diamondsuit$
\end{conjecture}

In fact, we can show that $c(2) = 4$ and $c(3) = 4$. \smallskip

Combining Conjecture \ref{conj11.1} with Corollary \ref{cor11.13}, one could arrive to an estimate of \hfill\break $\#(\mathsf{broken}_{(n+1)}(v))$ from below for Morse-Smale gradient fields $v$ on closed hyperbolic $(n+1)$-manifolds.  

Fortunately, regardless of the validity of the conjecture, $\#(\mathsf{broken}_{(n+1)}(v))$ has a lower boundary, delivered by the normalized hyperbolic volume!  The beautiful proposition below  has been recently proven by Hannah Alpert \cite{A}. The proof involves the same circle of amenable localization techniques as in \cite{Gr1}, \cite{AK}, being applied to intricate stratifications of certain compactifications of  the unstable manifolds of the gradient $v$-flow.

\begin{theorem}[{\bf Alpert}]\label{th11.A} Let $f: M \to \R$ be a Morse function on a closed hyperbolic manifold $M$ of dimension $(n+1)$ and $v$ its gradient-like field. Assume that $v$ satisfies the Morse-Smale transversality condition. Then %, assuming the validity of Conjecture \ref{conj11.1}, there exists a universal constant $const(n) > 0$ such that 
$$\#(\mathsf{broken}_{(n+1)}(v)) \geq %const(n) \cdot 
\textup{Vol}(M).$$
\hfill $\diamondsuit$
\end{theorem}
\smallskip

We just have internalized the crucial role that non-amenable fundamental groups  $\pi_1(X)$ and $\pi_1(DX)$ play in delivering some lower bounds of the traversing complexities $\mathbf{tc}^{\mathsf{lex}}_{\mathsf{trav}}(X)$ and $\mathbf{tc}^{\mathsf{lex}}_{\mathsf{trav}}(DX)$. % {\bf check notations !!!!!}

Now we will connect the complexity of another fundamental group $\pi_1(X/\d X)$ with the number of \emph{minimal} connected components in the $\Omega^\bullet$-stratification $\{X(v, \omega)\}_\omega$ of $X$. The considerations to follow are rather elementary; in particular, the amenability properties do not play any role here.\smallskip

Each $v$-trajectory $\g$ of the combinatorial type $\omega$ defines a loop or a bouquet  of loops in the quotient space $X/ \d X$. Similarly, each $\g$ defines its double $D(\g) \subset DX$.  The double $D(\g)$ is a chain of loops (like ``$\infty$" or  ``ooo"), the number of loops in the chain being equal to $\#(\g \cap \d X) -1$.

Therefore each $\g$ produces an element $[\g] \in \pi_1(X/ \d X)$ and a subgroup $[[\g]]$ of $\pi_1(X/ \d X)$, equipped with the ordered set of $(|\sup(\omega)| -1)$ generators---the $v$-ordered loops of the bouquet $[\g]$. Here $|\sup(\omega)|$ is the cardinality of the set $\g \cap \d X$. The element $[\g]$ and the subgroup $[[\g]]$ are constant within each connected component $X(v,  \s)$ of the pure stratum $X(v, \omega)$. This follows from the fact that the finite covering $$\Gamma: X(v,  \s) \cap \d X \to \Gamma\big(X(v,  \s) \cap \d X\big)$$ is trivial (\cite{K3}, Corollary 5.1). %\ref{cor7.1}
Let us denote by $\mathcal S^\bullet(v)$ the poset whose elements are the connected components $X(v, \s)$.

Therefore, we get a system of groups $\{[[\g_\s]]\}_{\s \in \mathcal S^\bullet(v)}$, linked by homomorphisms $\psi_{\s, \s'}: [[\g_\s]] \to [[\g_{\s'}]]$ for any pair $\s \succ \s'$ in $\mathcal S^\bullet(v)$.

This construction leads to the following 

\begin{lemma}\label{lem11.3} For a traversally generic field $v$ on $X$, the subgroups $\big\{[[\g_\s]]\big\}_{_{\s \in \mathcal S^\bullet_{\min}(v)}}$ generate $\pi_1(X/ \d X)$.
\end{lemma}

\begin{proof} We need to show that any loop $\rho: I/\d I \to X/ \d X$ through the point $\star = \d X/ \d X$ is a product of loops of the form the groups $[[\g_\s]]$, where $\s$ is a minimal element of the poset $\mathcal S^\bullet(v)$. Equivalently, it will suffice to show that any path $(\rho, \d\rho): (I, \d I) \to (X, \d X)$ can be homotoped, relatively to the boundary $\d X$, to an ordered union several \emph{segments} of oriented trajectories $\g$, labeled with the elements of the minimal set $\mathcal S^\bullet_{\min}(v)$. Here the segments of trajectories $\g$ are bounded by two points from $\g \cap \d X$ and the orientations of segments are not necessarily the ones induced by $v$. 

Note that each oriented segment of $\g$ belongs to the subgroup $[[\g]]$.

First, with the help of the $(-v)$-flow, we homotop $(\rho, \d\rho)$ to a path that is realized by several segments of trajectories, intermingled with some paths residing in $\d X$. 

The rest of argument is an induction based on the order $\succ$ in $\mathcal S^\bullet(v)$. We can replace any segment of $\g \subset X(v, \s)$ with the corresponding segment of any other trajectory from that stratum. This replacement is possible since $\Gamma: X(v, \s) \cap \d X \to \mathcal T(v, \s)$ is a trivial covering. Moreover, we can homotop such segment of $\g$ to a segment of some trajectory $\g'$, residing in any stratum $X(v, \s')$ that belongs to the closure of $X(v, \s)$. In other words, we can replace any segment of $\g \subset X(v, \s)$ with some segment of $\g' \subset X(v, \s')$, where $\s' \prec \s$. Following these replacements, eventually we will arrive to a ordered finite collection of segments of $\{\g_\s \subset X(v, \s)\}_{\s \in \mathcal S^\bullet_{\min}(v)}$; the collection will  represent the relative homotopy class of the path $\rho$ (some of the  the segments of trajectories $\g_\s$'s in this representation may appear with integral multiplicities). 
%WORK
\end{proof}

We denote by $c_{\mathsf{gen}}(\pi)$ the minimal number of generators in finite presentations of the group $\pi$.

\begin{theorem}\label{th11.13} For a traversally generic field $v$ on a connected compact smooth manifold $X$ with boundary, $$c_{\mathsf{gen}}\big(\pi_1(X/ \d X)\big) \leq \sum_{\s \in \mathcal S^\bullet_{\min}(v)} \; \Big(|\sup(\omega(\s))|-1\Big),$$
where  $\omega(\s)$ is the combinatorial type of a typical $v$-trajectory passing through the connected component labeled by $\s$ .
\end{theorem}

\begin{proof} By Lemma 4.1, %\ref{lem11.3} 
any element $\beta \in \pi_1(X/ \d X)$ produces a word in the alphabet whose letters are the loops, generated by the segments of special trajectories that label the minimal strata of  $ \mathcal S^\bullet_{\min}(v)$; each minimal stratum $X(v, \s)$, $\s \in \mathcal S^\bullet_{\min}(v)$, contains a unique special trajectory. The number of loops-letters in each trajectory is $|\sup(\omega(\s))|-1$. Therefore the total number of letters in the alphabet is given by the RHS of the formula above.
\end{proof}

\begin{corollary}\label{cor11.15} For a traversally generic field $v$ on a connected compact smooth $(n+1)$-dimensional manifold $X$ with boundary, the number of minimal connected components in $\mathcal T(v)$ satisfies the inequality:
$$\#\big( \mathcal S^\bullet_{\min}(v)  \big) \geq c_{\mathsf{gen}}\big(\pi_1(X/\d X)\big) /(n +1).$$

In particular, if $\d X$ consists of $m$ components, then $$\#\big( \mathcal S^\bullet_{\min}(v)  \big) \geq (m-1)/(n+1).$$
\end{corollary}

\begin{proof} Note that for a traversally generic $v$, $|\sup(\omega(\s))| \leq n + 2$ , so that the number of segments in which a typical trajectory $\g_\s$ is divided by $\d X$ is $n +1$ at most. Utilizing the arguments from Lemma 4.1, % \ref{lem11.3}, 
Theorem 4.9 %\ref{th11.13} 
implies the first inequality of the corollary.

To prove the second one, note that $X/ \d X$  admits a continuous  map $F$ onto a connected graph $G$ with two vertices, $a$ and $b$, and  $m$ edges. Indeed, let $U$ be a collar of $\d X$ in $X$.  We send $X \setminus \textup{int}(U)$ to $a \in G$, $\d X$ to $b$, and each component of the collar $U$ to the corresponding edge.  Each basic loop in $G$ lifts against $F$ to a loop in $X/ \d X$. Therefore  the exists an epimorphism $\pi_1(X/ \d X) \to \mathsf F_{m-1}$, where $\mathsf F_{m-1}$ denotes the free group in $m -1$ generators.   Thus $c_{\mathsf{gen}}\big(\pi_1(X/\d X)\big) \geq m -1$.
\end{proof}

\noindent{\bf Remark 4.5.}
%\begin{remark}\label{rem11.10} 
Note that, by definition, $\#\big( \mathcal S^\bullet_{\min}(v)  \big) \geq \#\big(\mathcal T(v)_{-n } \big)$. Thus, under the hypotheses of Theorem  \ref{th11.12} and by that theorem,  there exists an universal positive constant such that $\#\big( \mathcal S^\bullet_{\min}(v)  \big) \geq const(n)\cdot \textup{Vol}(M)$, where the hyperbolic volume $\textup{Vol}(M)$ depends only on $\pi_1(M)$. 

If $X$ is obtained from $M$ by deleting a single ball,  $\pi_1(M) \approx \pi_1(X/\d X)$. In such a case,  $\#\big( \mathcal S^\bullet_{\min}(v)  \big) \geq c_{\mathsf{gen}}\big(\pi_1(M)\big) /(n+1).$ So, for the hyperbolic $X = M \setminus D^{n+1}$, both lower bounds for $\#\big( \mathcal S^\bullet_{\min}(v)  \big)$ are expressed essentially  in terms of $\pi_1(M)$.
\hfill $\diamondsuit$
%\end{remark}

\bigskip

{\it Acknowledgments.} The author is grateful to Larry Guth and Hannah Alpert for the enjoyable in-depth discussions that have led to this work. He also likes to thank the referee whose advice helped to improve substantially the quality of this text.

\end{document}